\documentclass[12pt]{article}
\usepackage{graphicx}
\usepackage{xcolor}

\usepackage{amsmath, amsthm}
\usepackage{graphicx}

\usepackage{hyperref}
\usepackage{float}

\newtheorem{theorem}{Theorem}
\newtheorem{proposition}[theorem]{Proposition}
\newtheorem{lemma}[theorem]{Lemma}
\newtheorem{corollary}[theorem]{Corollary}
\newtheorem{definition}[theorem]{Definition}

\title{Corner Rectangle Visibility Graphs}

\author{Juni L. DeYoung \\ \small{Mathematics Department} \\ \small{Willamette University}  \\ \small{Salem, OR} \and Joshua D. Laison \\ \small{Mathematics Department} \\ \small{Willamette University}  \\ \small{Salem, OR} \and Junyi Li \\ \small{Department of Electrical \& Systems Engineering} \\ \small{Washington University in St.~Louis} \\ \small{St.~Loius, MO} \and Lani Southern \\ \small{Department of Mathematics} \\ \small{UC San Diego} \\ \small{La Jolla, CA}}

\date{}

\begin{document}

\maketitle

\begin{abstract}
We introduce corner rectangle visibility graphs (CRVGs), a combination of two geometrically defined classes of graphs: rectangle visibility graphs (RVGs) and rectangle-of-influence graphs (RIGs). 
A CRVG has vertices represented by axis-parallel rectangles in the plane, and edges represented by axis-parallel rectangles with one corner at a corner of a vertex-rectangle, an opposite corner at the boundary of another vertex-rectangle, and no vertex-rectangles in their interiors.  
We also consider CRVGs that only see in one or two directions (south CRVGs and southwest CRVGs).

We prove that south CRVGs have at most $\left[\frac{n^2}{4}\right]+n-2$ edges, and this bound is tight. This is the same as the tight edge bound for closed RIGs, but they are different graph classes. We also show that southwest CRVGs have at most $\left[\frac{n^2}{3}+\frac{n}{3}\right]-1$ edges, and this bound is tight.  We prove that CRVGs on $n$ vertices have at most $e$ edges, where $\lfloor \frac{3n^2}{8} \rfloor \leq e \leq \lfloor \frac{2n^2}{5} \rfloor$. Finally, we classify several families of graphs as CRVGs, SCRVGs, and SWCRVGs.

\end{abstract}

\section{Introduction}

\subsection{Rectangle Visibility Graphs}

\begin{definition}
    Let $R$ be a set of rectangles in the plane which have axis-parallel sides and do not intersect. Two rectangles are said to \textbf{see} each other if there is a vertical or horizontal band of positive width which intersects both rectangles and no other rectangles in $R$. Construct a graph $G$ with a vertex for each rectangle in $R$ and an edge between two vertices if their corresponding rectangles see each other. Then $G$ is said to be a \textbf{rectangle visibility graph}, or \textbf{RVG}, and $R$ is a \textbf{rectangle visibility representation} of $G.$  An example of an RVG and a rectangle visibility representation of it are shown in Figure~\ref{fig:maximal-RVG}.
\end{definition}

Rectangle visibility graphs and several variations on them have been well studied. Since Wismath proved that every planar graph is a rectangle visibility graph in 1989 \cite{wismath1989bar}, many authors have studied non-planar RVGs.  A graph $G$ has \textbf{thickness} $k$ if $k$ is the smallest integer such that $G$ is the union of $k$ planar graphs \cite{westgraphtheory}.  In 1999 Hutchinson, Shermer, and Vince proved that rectangle visibility graphs have thickness at most two, and also gave a bound on the number of edges in an RVG: a rectangle visibility graph on $n$ vertices has at most $6n-20$ edges, and the bound is tight for $n \geq 7$ 
\cite{hutchinson1999representations}.  As a corollary, they also obtained that $K_8$ is the largest complete RVG.  Figure~\ref{fig:maximal-RVG} shows a maximal RVG with 9 vertices, and deleting vertex $i$ in this graph yields an RVG representation of $K_8$.  The black dashed lines show some of the bands of sight between rectangles, and the red dashed line shows the only missing edge in the graph. In 1997, Dean and Hutchinson proved that a bipartite graph $K_{p,q}$ with $p \leq q$ is an RVG if and only if $p \leq 4.$  Dean and Hutchinson also proved that a bipartite RVG on $n$ vertices has at most $4n-12$ edges \cite{dean1997rectangle}.  Several variations of RVGs have been studied, including visibility of 3-dimensional boxes~\cite{bose99} and visibility of polyomino-like shapes~\cite{digiacomo18}.  Note that by Proposition~\ref{complete_bipartite} below, the corner rectangle visibility graphs defined in this paper have arbitrarily large thickness.

\begin{figure}
        \centering
        \includegraphics[height=4.5cm]{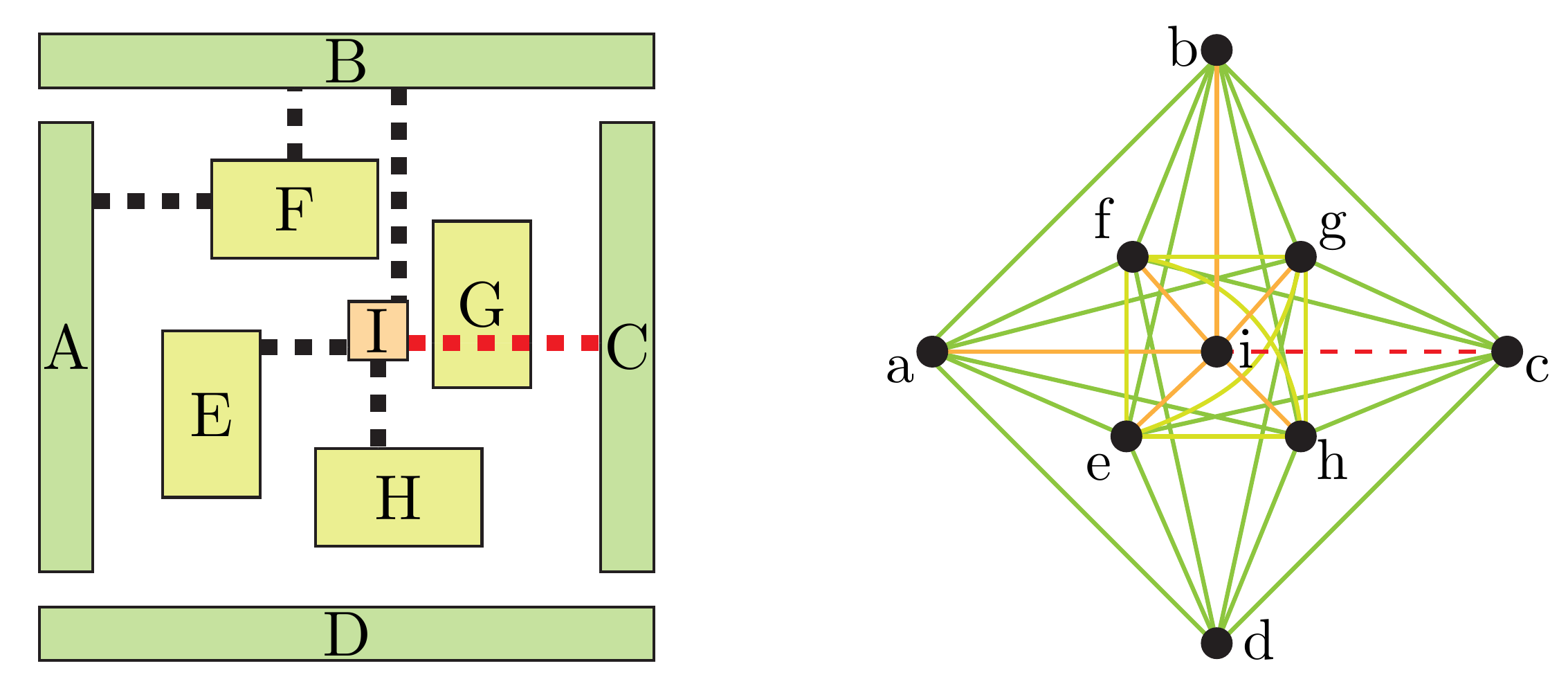}
        \caption{A maximal RVG on 9 vertices.  The black dashed lines show some of the bands of sight between rectangles, and the red dashed line shows the only missing edge in the graph.}
\label{fig:maximal-RVG} 
\end{figure}

\subsection{Rectangle of Influence Graphs}

\begin{definition}
     Given a set of points $P$ in the plane, and two points $A$ and $B$ in $P$, the \textbf{rectangle of influence} $R_{AB}$ between them is the rectangle with axis-parallel sides which has $A$ and $B$ on opposite corners. If $R_{AB}$ contains no other points in $P$, $A$ and $B$ are said to be \textbf{separated}.  If $R_{AB}$ contains its boundary then $A$ and $B$ are separated by a \textbf{closed} rectangle.  If $R_{AB}$ doesn't contain its boundary then $A$ and $B$ are separated by an \textbf{open} rectangle.
    \end{definition}

    \begin{definition}
    Let $P$ be a set of points in the plane, and let $G$ be a graph whose vertices correspond to points of $P$ in which vertices $a$ and $b$ are adjacent if and only if their corresponding points $A$ and $B$ are separated. Then $G$ is said to be a (\textbf{strong}) \textbf{rectangle of influence graph} (\textbf{RIG}).  If the rectangles of separation in $P$ are open, then $G$ is an \textbf{open} RIG, and if they're closed then $G$ is a \textbf{closed} RIG. 
    \end{definition}

    In our figures, we denote an open rectangle of influence with dashed lines and a closed rectangle of influence with solid lines. 
    For example, Figure \ref{fig:ex-rig} shows a closed RIG representation of $C_4$ and an open RIG representation of $K_5.$  With open rectangles, the representation on the left would have two extra pairs of separated points, and would yield $K_4$.  

\begin{figure}
        \centering
        \includegraphics[height=3cm]{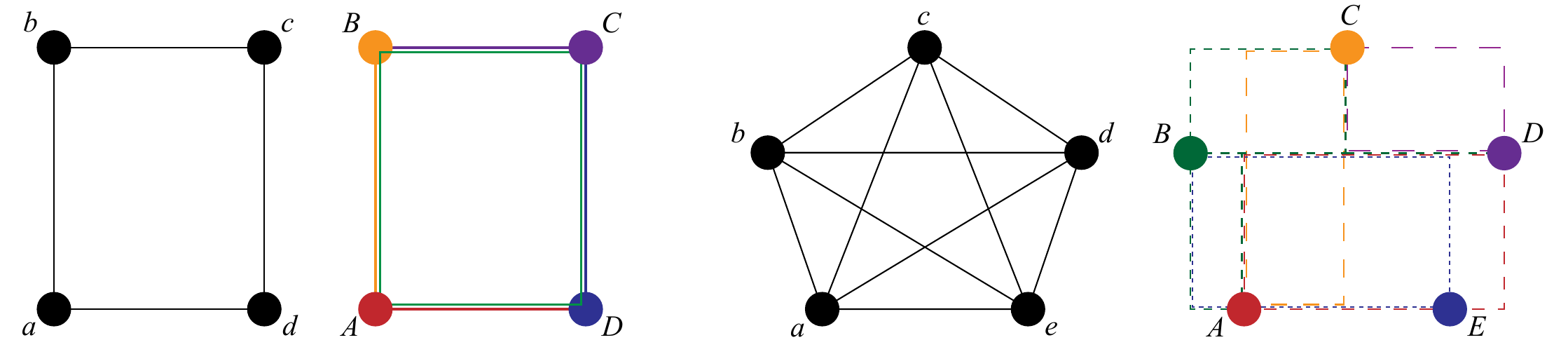}
        \caption{A closed RIG representation of $C_4$ and an open RIG representation of $K_5$.}
\label{fig:openrig}  \label{fig:closedrig} \label{fig:ex-rig}
\end{figure}

We now state a theorem bounding the number of edges in closed RIGs which is of particular interest to us. Note that in Theorem \ref{RIG_edge_bound}, $\left[\frac{n^2}{4}\right]$ means $\frac{n^2}{4}$ rounded to the nearest integer. In this case, it is equivalent to the floor function, but we keep the notation used in \cite{standardboxes}.

\begin{theorem} \label{RIG_edge_bound} (Alon, F{\"u}redi, Katchalski 1985 \cite{standardboxes}) A closed RIG on $n$ vertices has at most $\left[\frac{n^2}{4}\right]+n-2$ edges, and this bound is tight.
\end{theorem}

By Theorem~\ref{RIG_edge_bound}, closed RIGs have an edge bound which is quadratic in the number of vertices, while the RVG edge bound is linear by the result of Hutchinson, Shermer, and Vince \cite{hutchinson1999representations}. We will follow the proof of Theorem \ref{RIG_edge_bound} to prove our edge bounds on SCRVGs and SWCRVGs in Sections~\ref{sec:SCRVGs} and \ref{sec:SWCRVGs}. 

Liotta et al classified several families of graphs as representable or non-representable using RIG representations \cite{liotta1998rectangle}. We discuss some of these in Section~\ref{classification}.

\subsection{Corner Rectangle Visibility Graphs}

Let $R$ be a set of rectangles in the plane which have axis-parallel sides and do not intersect.  For each rectangle, choose one of its corners to be an \textbf{eye}.  Rectangles with designated eyes are called \textbf{corner rectangles}. We represent the eye of a rectangle with an arrow pointing in its direction of sight. We use the four cardinal directions to label the corners of each rectangle. The \textbf{north} corner is its upper left corner, and the directions proceed \textbf{west}, \textbf{south}, and \textbf{east}, counterclockwise around the rectangle. A corner rectangle whose eye is its north corner is a $\textbf{north rectangle}.$ South, east, and west rectangles are similarly defined. We will also sometimes refer to \textbf{$D$-rectangles}, where $D$ is an unspecified direction from $\{N,W,S,E\}$.  We label the left, bottom, right, and top sides of the rectangle NW, SW, SE, and NE respectively, to indicate which corners they are between. We denote the north corner of a rectangle $A$ by $A_N$ and its $x-$coordinate by $A_N(x),$ and similarly denote the other corners of $A$ and their $x-$ and $y-$coordinates.

A corner rectangle $A$ \textbf{sees} a corner rectangle $B$ if there is a non-degenerate closed rectangle of influence intersecting $A$ exactly in its eye, $B$ in its boundary, and not intersecting any other rectangle in $R$, as shown in Figure~\ref{fig:crvg-definition}.  In this figure, note that rectangle $A$ sees only rectangles $D$ and $F$; the others are blocked from view.  In particular, the rectangle of influence from $A$ to $I$ is blocked by $H$, and the rectangle of influence from $A$ to $C$ is blocked by $A$ itself.

Equivalently, we can also think of $A$ as looking in the direction of a quarter-plane at its eye, and seeing rectangles in this quarter-plane not blocked by other rectangles casting quarter-plane shadows, as shown in Figure~\ref{fig:crvg-shadow}.

\begin{figure}
        \centering
        \includegraphics[height=7cm]{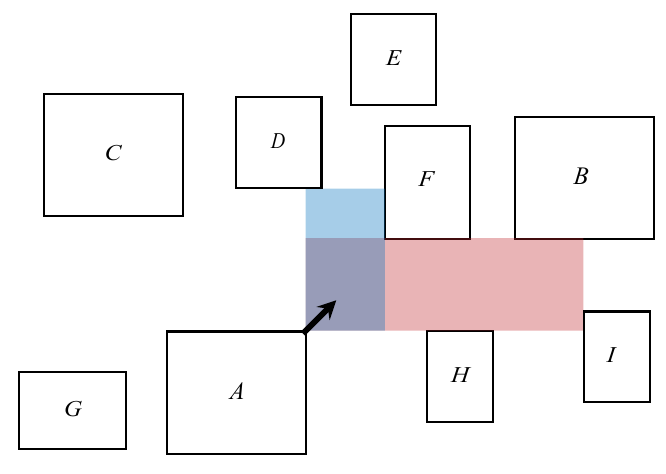}
        \caption{Examples of sight and non-sight between corner rectangles.  Rectangle $A$ can only see rectangles $D$ and $F$.}
\label{fig:crvg-definition} 
\end{figure}

\begin{figure}
        \centering
        \includegraphics[height=7cm]{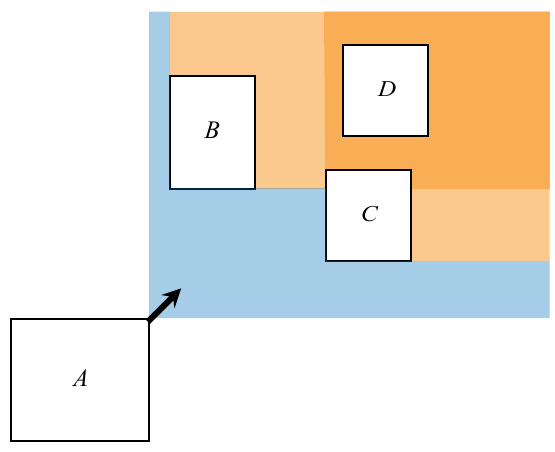}
        \caption{The shadows from rectangles $B$ and $C$ block $D$ from being visible from $A$.}
\label{fig:crvg-shadow} 
\end{figure}

Given a set of corner rectangles $R$, we construct a graph $G$ with a vertex for each rectangle in $R$ and an edge between vertices $a$ and $b$ if and only if for their corresponding rectangles $A$ and $B$ in $R,$ $A$ sees $B$ or $B$ sees $A.$ We say that $G$ is a \textbf{corner rectangle visibility graph (CRVG)} and $R$ is its \textbf{corner rectangle visibility representation (CRV representation).} We denote rectangles with uppercase letters and vertices with their corresponding lowercase letters.

\section{D-Monotone Sequences}

We are sometimes interested in what configurations of rectangles can be seen by a rectangle looking in a particular direction. This motivates the following definitions.

\begin{definition} 

A set of points $P_1=(x_1,y_1), \ldots, P_n=(x_n,y_n)$ is \textbf{monotonically increasing} if the function $f(x_i)=y_i$ is a strictly increasing function.
Similarly, a set of points is \textbf{monotonically decreasing} if $f(x_i)=y_i$ is a strictly decreasing function.

Let $\{R_i\}$ be a set of rectangles in the plane and let $D$ be one of the four directions (north, south, east, or west).
If $D$ is north or south and the set of points $\{R_{i_D}\}$ is monotonically increasing, or if $D$ is east or west and $\{R_{i_D}\}$ is monotonically decreasing, then we say that $\{R_i\}$ is \textbf{D-monotone.} In particular, if we know that $D$ is north, south, east, or west, we  say $\{R_i\}$ is \textbf{$N$-monotone}, \textbf{$S$-monotone}, \textbf{$E$-monotone}, or \textbf{$W$-monotone}, respectively.  Figure~\ref{fig:d monotone tree} shows an example of an $N$-monotone set of rectangles.  We may also say that $\{R_i\}$ is $-D$-monotone, which is monotone in the opposite direction of $D$. The difference between north/south and east/west is so that the neighborhood of a $D$-rectangle is $-D$-monotone, as described in Lemma \ref{monotone-out-nbhd}.
\end{definition}

\begin{proposition}
\label{trees}
Every tree has a $D$-monotone representation of $-D$-rectangles.
\end{proposition}

\begin{proof}
The representation can be built recursively. Without loss of generality, let $D$ be north. Pick a root for the tree, $a.$ Place the rectangle $A$ at the bottom of the construction. Place its children $B,$ $C,$ $D,...$ above it in an $N-$monotone manner, where each rectangle extends to the right of the one below it, but never as far as the root, as shown in Figure~\ref{fig:d monotone tree}. Then all the children of $A$ can see $A.$ Now repeat this process recursively with the trees formed by taking each of the children of $a$ as the root of a new tree. 
\end{proof}

\begin{figure}
        \centering
        \includegraphics[height=6cm]{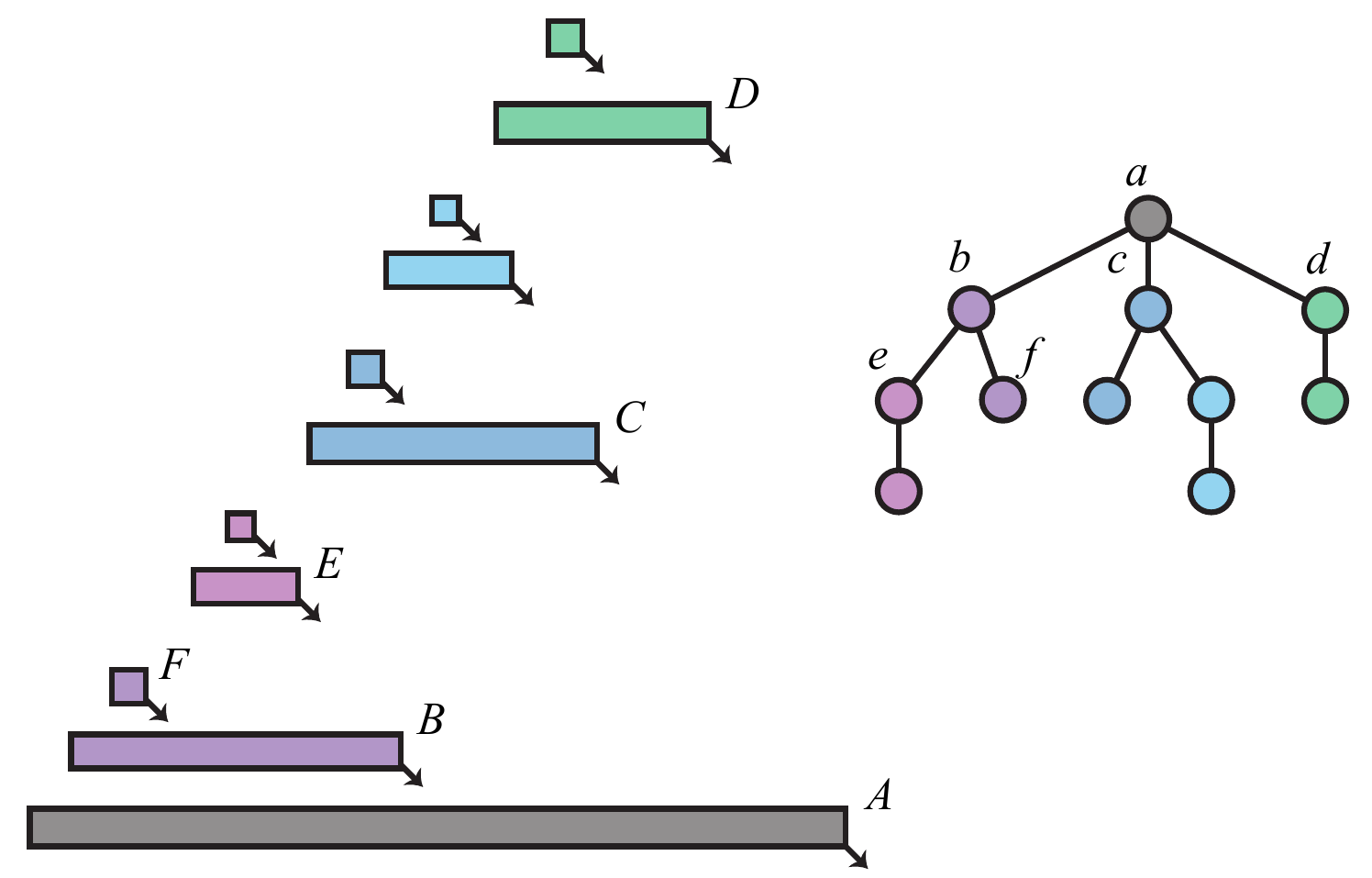}
        \caption{An $N-$monotone set of south rectangles in a representation of a tree.}
\label{fig:d monotone tree} 
\end{figure}

Going forward, it will sometimes be useful to consider the directed graph of a representation, where there is an edge from $a$ to $b$ exactly when the rectangle $A$ sees the rectangle $B.$ 

Recall that in a directed graph, the out-neighborhood $N_+(a)$ of a vertex $a$ is the set of vertices with a directed edge from $a$.  In a directed CRVG, the out-neighborhood $N_+(A)$ of a rectangle $A$ is the set of rectangles which can be seen by $A.$  The neighborhood $N(a)$ of a vertex $a$ in a directed or undirected graph is the set of vertices adjacent to $a$.

\begin{lemma} \label{monotone-out-nbhd}
The out-neighborhood of a $D$-rectangle $A$ in a CRV representation is a $-D$-monotone set of rectangles.
\end{lemma}

\begin{proof}

    Without loss of generality, let $A$ be a south rectangle. Let $B_1,$ $B_2,$~$\ldots,$~$B_n$ be the elements of $N_+(A)$ ordered by increasing $x$-coordinate of their north corners. Note that to be a $N$-monotone sequence, the north corners of $B_1,$~$\ldots,$~$B_n$ must have increasing $y$-coordinate as well.

    Suppose by way of contradiction that $B_1,\ldots, B_n$ is not $N$-monotone, that is, their $y$-coordinates are not increasing. Then there are some rectangles $B_i,$ and $B_k$ in $N_+(A)$ with $i<k$ such that $B_{i_N}(y) \geq B_{k_N}(y).$  Possible arrangements of the three rectangles are shown in Figure~\ref{monotone_figure}. Then any rectangle of influence from the south corner of $A$ to $B_k$ must intersect $B_i,$ which contradicts the fact that $A$ can see $B_k.$
\end{proof}

\begin{figure}    
    \centering
    \includegraphics [width=7cm]{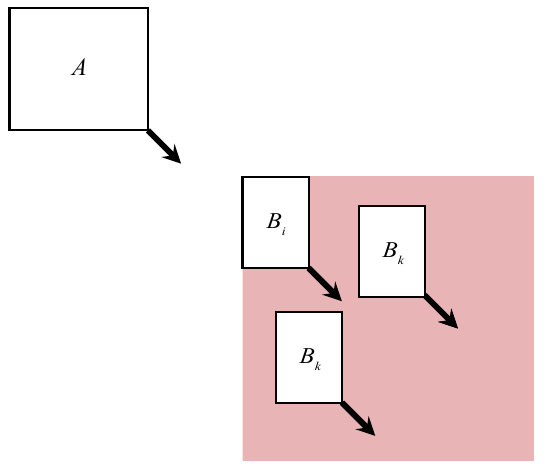}
    \caption{In Lemma~\ref{monotone-out-nbhd}, the south rectangle $A$ cannot see $B_k$ if the north corner of $B_k$ is anywhere in the shaded area.} \label{monotone_figure}
\end{figure}

\begin{definition}
For a $D$-rectangle $A$, $B$ is \textbf{side-visible} to $A$ if $A$ sees $B$ but $A$ doesn't see the $-D$ corner of $B$. An example of a south rectangle $A$ and a rectangle $B$ which is side-visible from $A$ is shown in Figure \ref{sidevis2}.
\end{definition}

\begin{figure}
    \centering
    \includegraphics [width=8cm]{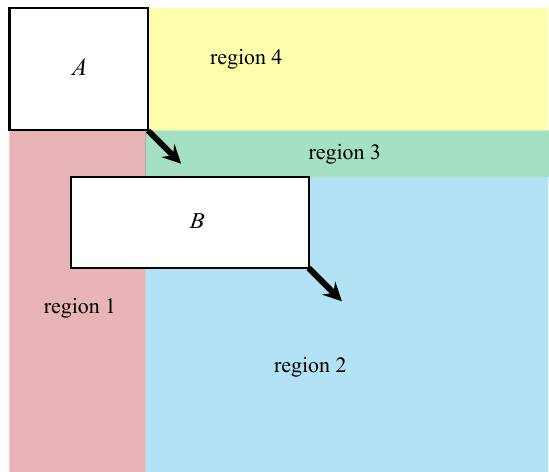}
    \caption{A rectangle $B$ side-visible from $A$, and the possible regions a second rectangle visible from $A$ could be in, as in Lemma~\ref{side-viz-lemma}.}
    \label{sidevis2}
\end{figure}

\begin{lemma}
    \label{side-viz-lemma}
    For a rectangle $A$, there are at most two rectangles that are side-visible from $A$.
\end{lemma}

\begin{proof}
Without loss of generality, let $A$ be a south rectangle.  We claim that there is only one rectangle $B$ side-visible from $A$, where $A$ sees the NE side of $B$. See Figure~\ref{sidevis2} for the placement of $B$ and consider the possible locations for the north corner of a rectangle $C.$ We show that there is nowhere to place the north corner of $C$ that results in both $B$ and $C$ being side-visible from $A$ to their northeast sides.

\begin{enumerate}
    \item If the north corner of $C$ is in region 1 then $A$ cannot see $C,$ unless $C$ extends into region 3, in which case it blocks sight from $A$ to $B.$
    \item If the north corner of $C$ is in region 2 then $A$ cannot see $C$.
    \item If the north corner of $C$ is in region 3 then $C$ is visible but not side-visible from $A.$
    \item If the north corner of $C$ is in region 4 then $A$ cannot see $C$ unless $C$ extends into region 3, in which case $C$ is side-visible from $A$ but to its NW side.
\end{enumerate}

Thus $A$ can have at most one rectangle which is side-visible from $A$ to its NE side and at most one rectangle side-visible from $A$ to its NW side.
\end{proof}

\begin{lemma}
\label{sight in a north-monotonic sequence}
A $D$-directional rectangle $A$ in a $-D$-monotone set of rectangles $X$ can see at most two other rectangles in $X$.
\end{lemma}

\begin{proof}
Without loss of generality, let $X$ be an $N$-monotone sequence of rectangles and $A$ be a south rectangle in $X.$ By definition of $N$-monotone, the north corners of all rectangles in $X$ are to the left or above the south corner of $A$. Thus, any rectangles in $X$ seen by $A$ must be side-visible from $A.$ By Lemma \ref{side-viz-lemma}, there are at most two rectangles side-visible from $A.$ Thus $A$ can see at most two other rectangles in $X.$
\end{proof}

\section{South Corner Rectangle Visibility Graphs}
\label{sec:SCRVGs}

In this section we consider graphs with CRV representations containing only south rectangles. We call these graphs \textbf{south corner rectangle visbility graphs} or \textbf{SCRVGs}. Of course, the results that follow also hold for any CRVGs whose rectangles all look in the same direction. 

 \begin{lemma} \label{lem:acyclic}
 A directed SCRVG is acyclic.
 \end{lemma}

\begin{proof}
Suppose $G$ is a directed SCRVG represented by a set of south rectangles, and suppose $a \to b$ is a directed edge from $a$ to $b$ in $G$. Then the rectangle $A$ sees $B$ from its south corner, so $A_S(x)<B_S(x)$ and $A_S(y)>B_S(y)$, so $B$ doesn't see $A$.  Similarly, if $a_1 \to a_2 \to \cdots \to a_k$ is a directed path in $G$, then $A_k$ doesn't see $A_1$. 
\end{proof}

\begin{proposition}
\label{SCRVG complete graphs}
The complete graph $K_n$ for $n \leq 4$ is an SCRVG. For $n \geq 5,$ $K_n$ is not an $SCRVG.$
\end{proposition}

\begin{proof}
    An SCRV representation of $K_4$ is shown in Figure~\ref{SCRVG K4}. Smaller complete graphs can be formed by removing some of the rectangles from Figure~\ref{SCRVG K4}. The fact that $K_n$ is not an SCRVG for $n \geq 5$ follows from Theorem \ref{SCRVG edge bound}, since $K_n$ has $\frac{n^2}{2} - \frac{n}{2}$ edges, which is more than $\left[ \frac{n^2}{4} \right]+n-2$ edges when $n \geq 5$.
\end{proof}

\begin{figure}
    \centering  \includegraphics[width=0.4\linewidth]{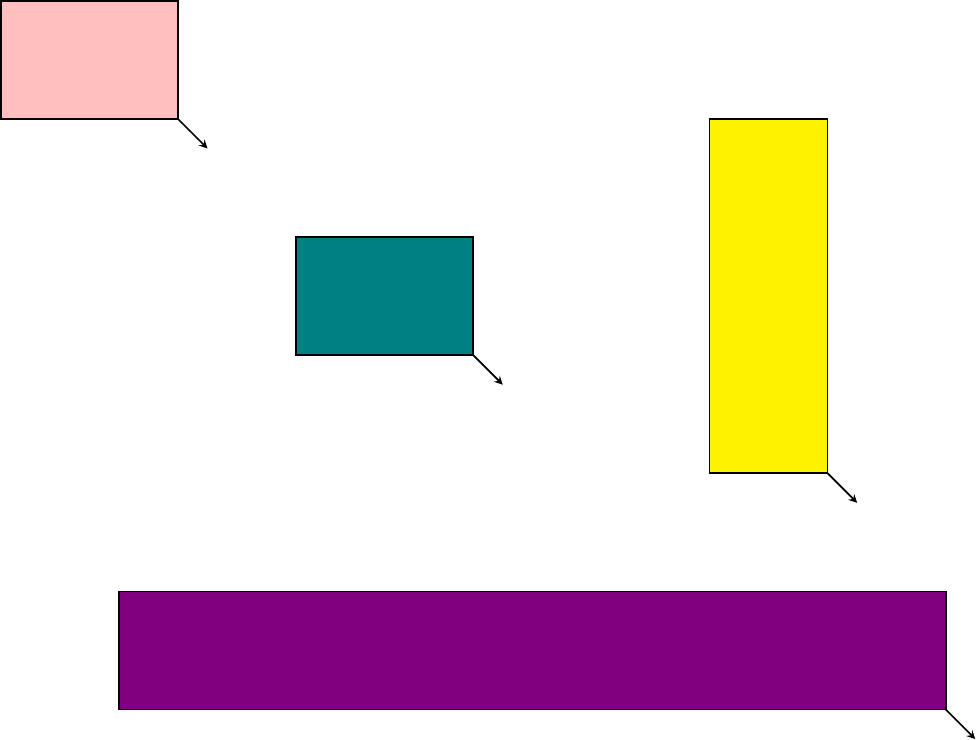}
    \caption{An SCRV representation of $K_4.$}  \label{SCRVG K4}
\end{figure}

\begin{proposition}
\label{complete_bipartite}
    The complete bipartite graph $K_{m,n}$ is an SCRVG for all integers $m,n \geq 1.$
\end{proposition}

\begin{proof}
    Let the independent sets of $K_{m,n}$ be $S_1$ and $S_2$, with $|S_1|=m$ and $|S_2|=n$. Using squares of equal size, construct an $N-$monotone set of $m$ south rectangles representing $S_1$, and an $N-$monotone set of $n$ south rectangles strictly below and to the right of the other set representing $S_2$, as shown in Figure~\ref{fig:SCRVG bipartite}. Then all the rectangles in $S_1$ can see all the rectangles in $S_2$ and none of the rectangles in $S_i$ can see any of the others in $S_i.$
\end{proof}

\begin{figure}
    \centering
    \includegraphics[width=5cm]{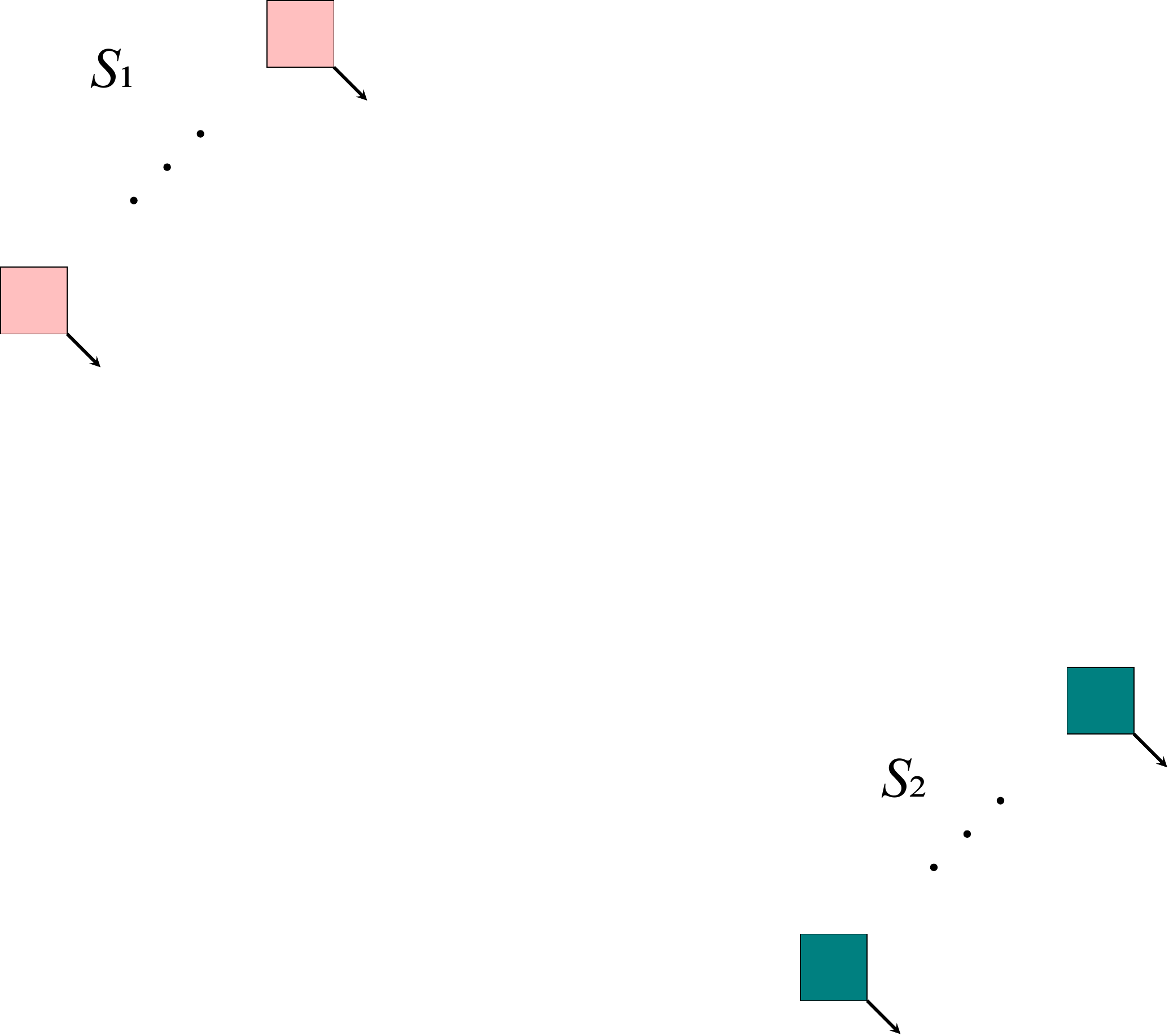}
    \caption{An SCRV representation of a complete bipartite graph.}
    \label{fig:SCRVG bipartite}
\end{figure}

We define the graph $K_{m,n}+P_n$ to be the complete bipartite graph $K_{m,n}$ with independent sets $S_1$ and $S_2$ of sizes $m$ and $n$, union with a path on the vertices of $S_2$.

\begin{proposition}
$K_{m,n}+P_n$ is an SCRVG for all integers $m,n \geq 1.$
\end{proposition}

\begin{proof}
Define $S_1$ and $S_2$ as above. We construct $S_1$ as in the previous construction but for $S_2$ make each rectangle in the path wider than the one above it so that the rectangle above it sees it, as shown in Figure \ref{fig:SCRVG bipartite with path}. Then all the rectangles in $S_1$ see all the rectangles in $S_2$, and the rectangles in $S_2$ also form a path.
\end{proof}

\begin{figure}
    \centering
    \includegraphics[width=5cm]{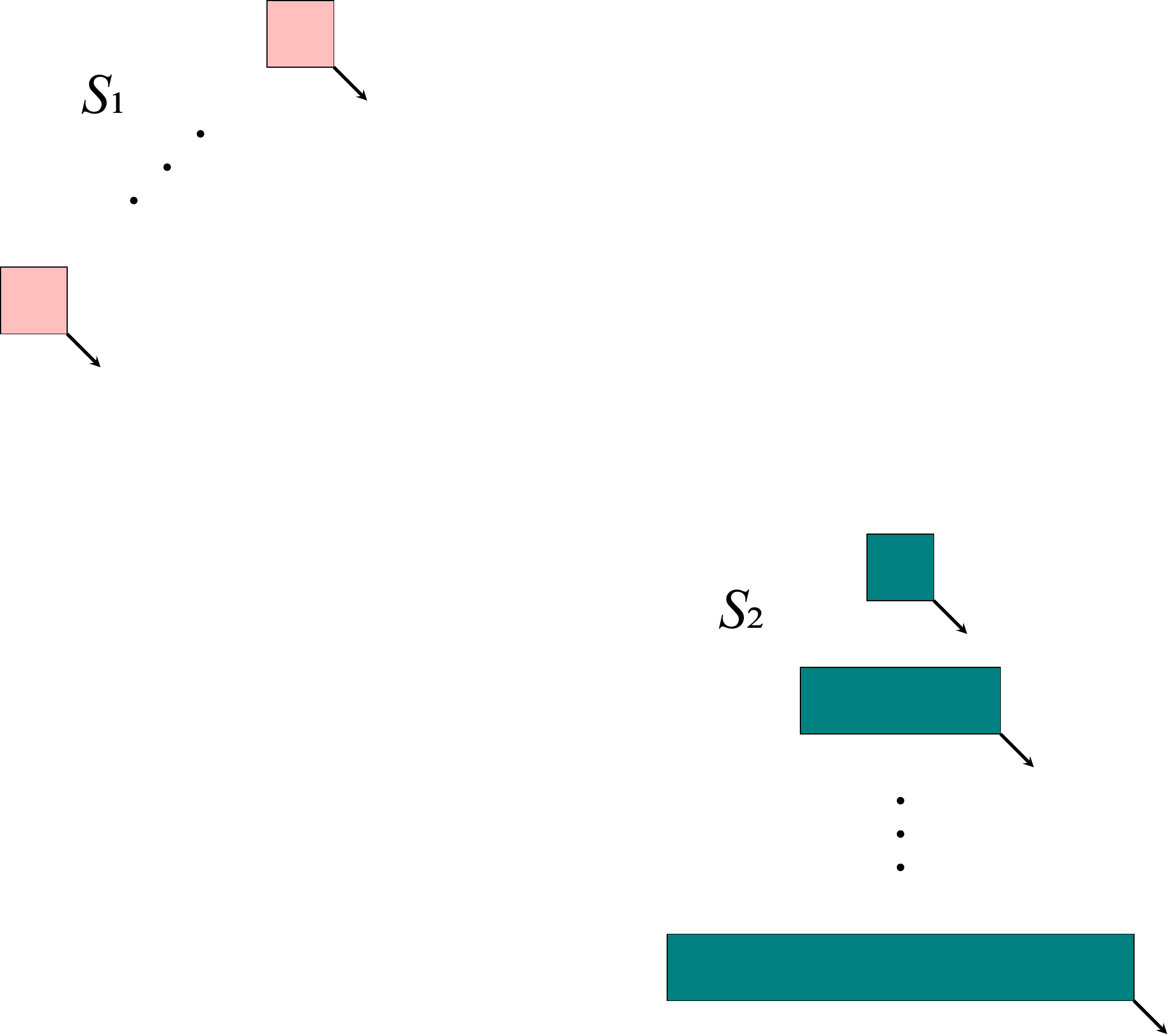}
    \caption{An SCRV representation of a complete bipartite graph with a path.}
    \label{fig:SCRVG bipartite with path}
\end{figure}

\begin{figure}
    \centering
    \includegraphics[width=4cm]{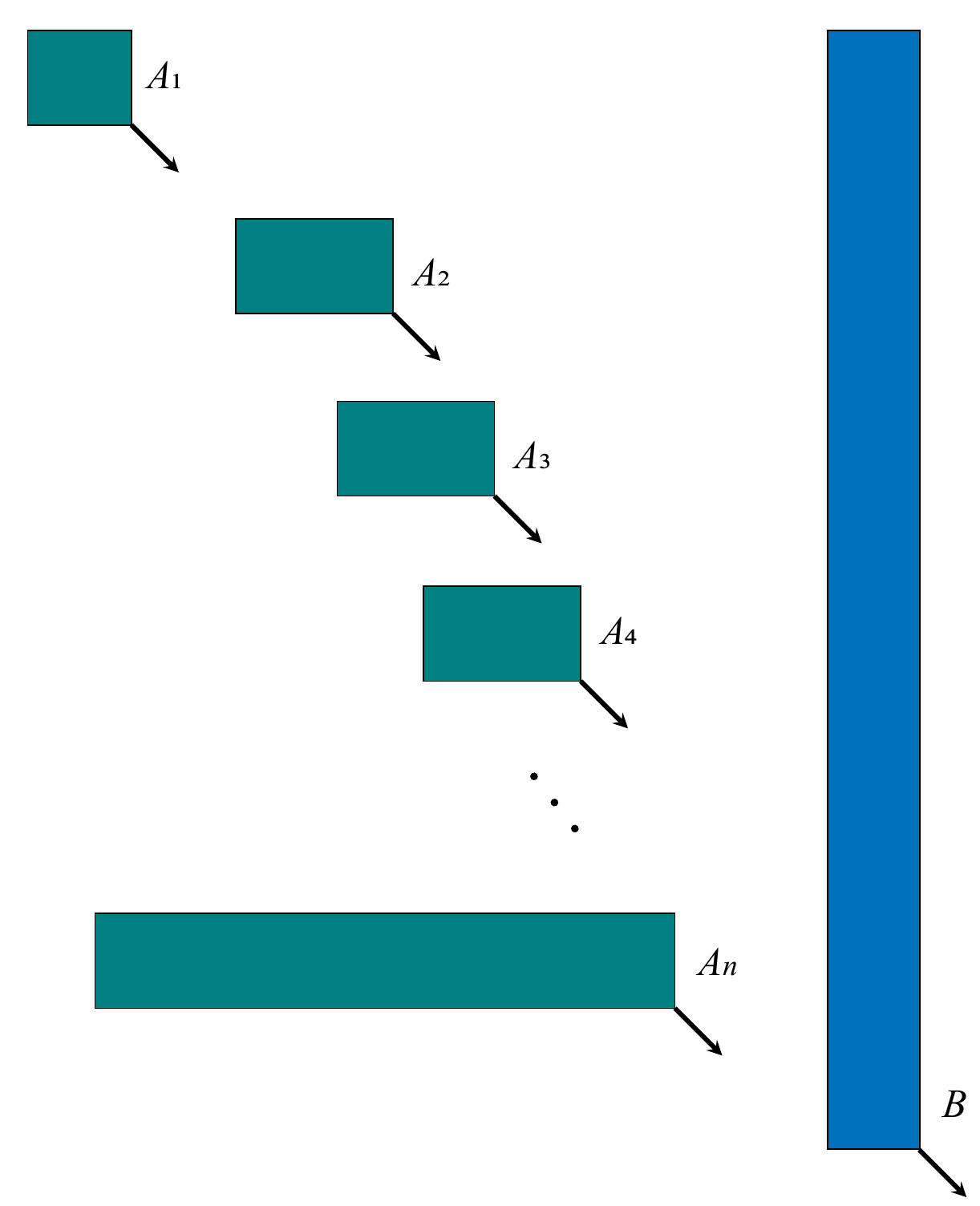}
    \caption{An SCRV representation of a cycle, and of a wheel using the additional rectangle $B$.}
    \label{fig:South-cycle}
\end{figure}
\begin{proposition}
    A cycle $C_n$ is an SCRVG. A wheel is also an SCRVG.
    \label{SCRVG_cycle}
\end{proposition}

\begin{proof}
First we construct an SCRV representation of $C_n$. Call the $n$ rectangles $A_1, \ldots, A_n.$ Place $A_1.$ Then place $A_2, \ldots A_{n-1}$ such that $A_{i+1}$ is below $A_i$ and $(A_i)_N(x)<(A_{i+1})_N(x)<(A_i)_E(x)<(A_{i+1})_E(x)$, as shown in Figure \ref{fig:South-cycle}. Place $A_n$ below $A_{n-1}$ such that $(A_n)_N(x)<(A_1)_S(x)<(A_{n-1})_N(x)<(A_{n-1})_S(x)<(A_n)_S(x).$

To make a wheel, add one rectangle $B$ to the right of all the rectangles in the cycle, so that $B_N(y)>(A_1)_S(y)$ and $B_S(y)< (A_n)_S(y),$ again as shown in Figure~\ref{fig:South-cycle}.

\end{proof}

We now prove an edge bound for SCVRGs in Theorem~\ref{SCRVG edge bound} below.  Our proof is adapted from the proof of Theorem \ref{RIG_edge_bound} from \cite{standardboxes}.

We note that Theorem~\ref{RIG_edge_bound} and Theorem~\ref{SCRVG edge bound} imply that SCRVGs and closed RIGs have the same edge bound.  However, the maximal constructions provided are different.  Closed RIGs and SCRVGs are not the same graph class, as trees with more than four leaves are not closed RIGS \cite{liotta1998rectangle} but any tree is an SCRVG by Proposition \ref{trees}.  We don't know whether all closed RIGs are SCRVGs.  We conjecture that the construction of a maximal closed RIG given in~\cite{standardboxes} is not an SCRVG.

\begin{theorem}
    \label{SCRVG edge bound}
    An SCRVG with $n$ vertices has at most
    \begin{equation}
    \left[ \frac{n^2}{4} \right]+n-2
    \label{scrvg_formula}
    \end{equation}    
    edges for $n \geq 2$, and this bound is tight.
\end{theorem}

\begin{proof}
 First we show by induction that an SCRVG on $n$ vertices can have no more than $\left[\frac{n^2}{4}\right]+n-2$ edges.
    
    The bound holds for $2,$ $3,$ and $4$ vertices because the number of edges in the complete graph is less than or equal to the number of edges given by the formula. 
    
    Let $n \geq 5,$ let $S$ be a set of $n$ south rectangles, and let $G$ be the graph represented by $S.$ By way of induction, suppose that all SCRVGs on $n-2$ vertices have at most $\left[\frac{(n-2)^2}{4}\right]+(n-2)-2$ edges.
    
    Let $A$ be a rectangle in $S$ whose eye has minimal $x-$coordinate. Let $B_1,$~$\ldots,$~$B_r$ be the rectangles in $N(A)$. Since $A$ is a rectangle with minimal $x-$coordinate, there can be no rectangles in $S$ that see $A.$ Thus $N(A)=N_+(A)$. By Lemma~\ref{monotone-out-nbhd}, $N(A)$ is an $N$-monotone sequence 
    
    By Lemma~\ref{lem:acyclic}, $G$ has no directed cycles. Thus, there must be some $B_k$ that is not seen by any other rectangle in $N(A).$  

    Let $T$ be the set of rectangles $S\backslash(\{A\} \cup N(A))$.  If a rectangle is in $T$ then it is not adjacent to $A$ but may have an edge to $B_k$. Each rectangle $C$ which is neither $A$ nor $B_k$ is either in $N(A)$ or in $T.$ If $C$ is in $N(A)$ then $C$ is adjacent to $A$, and if $C$ is in $T$ then $C$ is not adjacent to $A$ but may be adjacent to $B_k.$ Additionally, at most two of these rectangles are adjacent to both $A$ and $B_k$ by Lemma~\ref{sight in a north-monotonic sequence}. 
    Then the number of edges incident to either $a$ or $b_k$ in $G$ is at most $(n-2)+2+1=n+1$ (that is $n-2$ edges from $N_+(a) \cup T$,  $2$ edges from $b_k$ to other vertices in $N(a),$ and $1$ edge between $a$ and $b_k)$.

  Counting the edges in $G$ gives $|E(G)| \leq E(G \backslash \{a,b_j\})+n+1.$ By the induction hypothesis, $|E(G \backslash \{A,B_j\})| \leq \left[\frac{(n-2)^2}{4}\right]+(n-2)-2.$ Thus $|E(G)| \leq \left[\frac{(n-2)^2}{4}\right]+(n-2)-2+n+1=\left[\frac{n^2}{4}\right]+n-2.$

We now show the bound is tight. A representation of an SCRVG $G$ on $n$ vertices with $\left[\frac{n^2}{4}\right]+n-2$  edges is shown in Figure \ref{fig:max-scrvg}.
Label the sets $A,$ $B,$ and $C$ of $G$ as shown in the figure. Suppose there are $a$ rectangles in $A,$ $b$ rectangles in $B,$ and $1$ rectangle in $C.$
 The graph corresponding to $B$ is a series of triangles which have an edge from one set of triangles to the next.  When $b$ is not divisible by $3,$ we form $B$ by adding rectangles in the order shown in Figure \ref{fig:max-scrvg}. There are $\left \lfloor\frac{4}{3}b-1 \right \rfloor$ edges within $B.$ 
 $G$ has all possible edges between the sets $A$ and $B$ and $A$ and $C$, which is $ab$ and $a$ edges, respectively. Two out of every three vertices of a triangle in $B$ have an edge to the vertex in $C,$ which is $\left \lceil \frac{2}{3}b \right \rceil$ edges. Thus there are \begin{equation}\left \lfloor\frac{4}{3}b-1 \right \rfloor + \left \lceil \frac{2}{3}b \right \rceil+ab+a \label{formula:max_edges} \end{equation} edges in the graph $G$.

If $n=6x$ for some integer $x$, put $3x-1$ rectangles in $A,$ $3x$ rectangles in $B,$ and $1$ rectangle in $C.$ Using  Formula~\ref{formula:max_edges}, we have $\left \lfloor\frac{4}{3}(3x)-1 \right \rfloor + \left \lceil \frac{2}{3}(3x) \right \rceil+(3x-1)(3x)+(3x-1)=9x^2+6x-2$ edges. In terms of $n,$ this is $\frac{n^2}{4}+n-2$ edges, which is indeed $\left[ \frac{n^2}{4} \right]+n-2$ since in this case $n^2$ is divisible by $4.$

One can check that the remaining cases also satisfy Formula~\ref{scrvg_formula} using the distributions that follow: 
If $n=6x+1$ put $3x$ rectangles in each of sets $A$ and $B$ and $1$ rectangle in $C.$
If $n=6x+2,$ put $3x$ rectangles in $A,$ $3x+1$ rectangles in $B,$ and $1$ rectangle in $C.$ 
If $n=6x+3,$ put $3x$ rectangles in $A,$ $3x+2$ rectangles in $B,$ and $1$ rectangle in $C.$
If $n=6x+4,$ put $3x+1$ rectangles in $A,$ $3x+2$ rectangles in $B,$ and $1$ rectangle in $C.$
If $n=6x+5,$ put $3x+1$ rectangles in $A,$ $3x+3$ rectangles in $B,$ and $1$ rectangle in $C.$
\end{proof}

\begin{figure}[!ht]
      \centering
      \includegraphics[width=\textwidth]{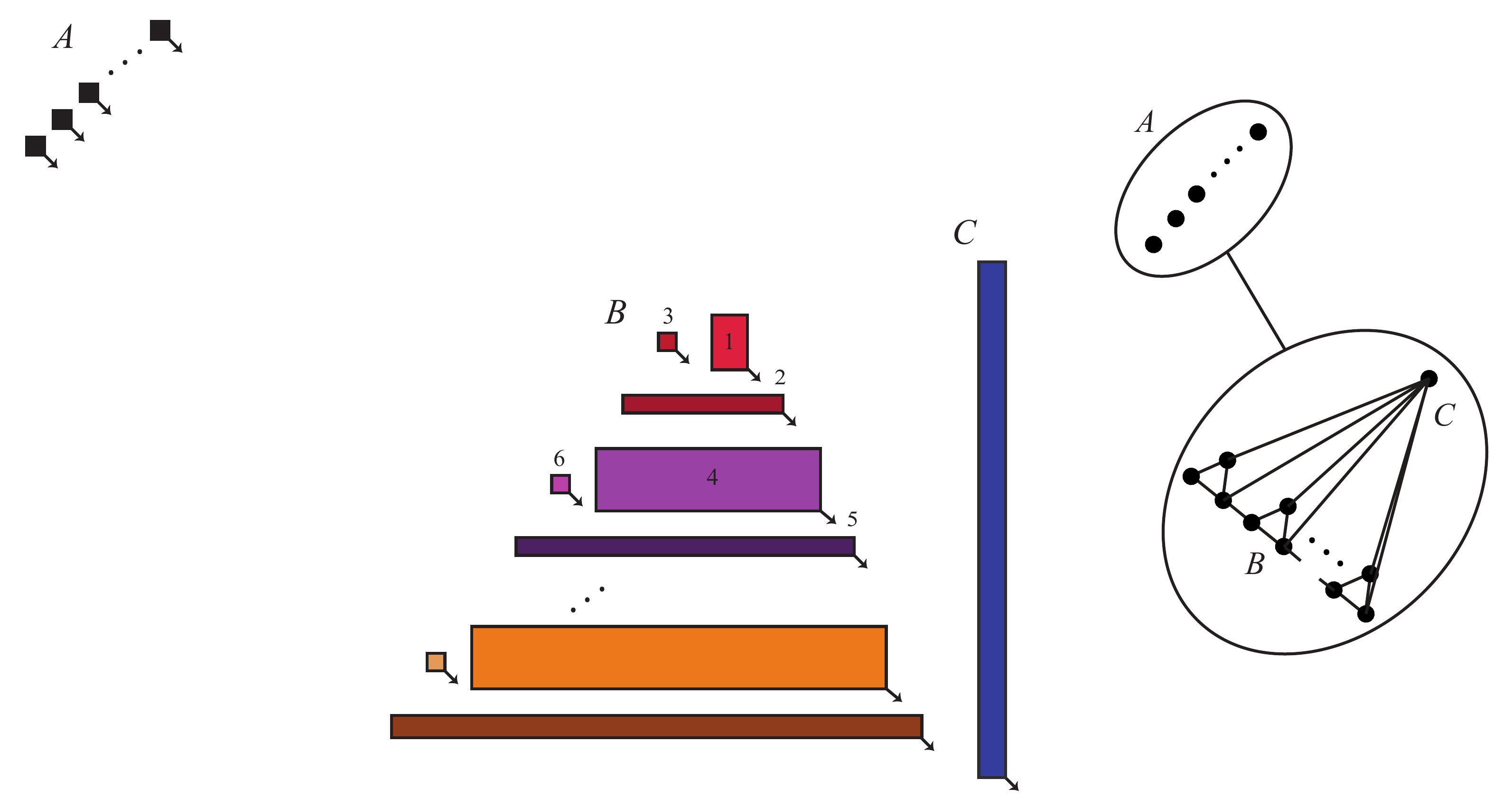}
      \caption{A representation of an SCRVG with $n$ vertices and $\left[\frac{n^2}{4}\right]+n-2$ edges.}
      \label{fig:max-scrvg}
\end{figure}

\section{Southwest Corner Rectangle Visibility Graphs}
\label{sec:SWCRVGs}
We now consider a natural extension from the previous section and allow representations with rectangles that are allowed to look in one of two perpendicular directions. In particular, we study representations of graphs which contain only south rectangles and west rectangles, \textbf{southwest corner rectangle visibility graphs}, or \textbf{SWCRVGs}. 

\begin{lemma}
A directed SWCRVG is acyclic. \label{SWCRVG-acyclic}
\end{lemma}

\begin{proof}
The proof is analogous to the proof of Lemma~\ref{lem:acyclic}.  Suppose $G$ is a directed SWCRVG represented by a set of south and west rectangles, and suppose $a \to b$ is a directed edge from $a$ to $b$ in $G$.  Either $A$ is a south rectangle or a west rectangle.  In both cases, $A_S(y) = A_W(y) >B_S(y)= B_W(y)$, so $B$ doesn't see $A$.  Similarly, if $a_1 \to a_2 \to \cdots \to a_k$ is a directed path in $G$, $A_k$ doesn't see $A_1$.
\end{proof}

We now classify a few families of graphs which are SWCRVGs. Of course, all SCRVGs are also SWCRVGs.

In the following proposition, we consider the complete tripartite graph with a path $K_{l,m,n}+P_n$.  This is the graph with vertices partitioned into sets $A$, $B$, and $C$ of sizes $l$, $m$, and $n$ respectively, all possible edges between these three sets, and such that the induced subgraph on $C$ is a path.

\begin{proposition}
   $K_{l,m,n}+P_n$ is an SWCRVGs for all positive integers $l,$ $m,$ and $n.$
\end{proposition}

\begin{proof}
    An example is shown in Figure~\ref{SW 3-partite}.
    To make $A,$ place $l$ squares looking south in a north monotone sequence. To make $B,$ place $m$ rectangles looking west that form both a north and west monotone sequence. Place $B$ completely below and to the right of $A.$ To make $C,$ place $n$ south facing rectangles so that they form both a north and an east monotone sequence. Place $C$ completely below $B$, to the right of $A$ and the left of $B.$
\end{proof}

\begin{figure}
    \centering
    \includegraphics[width=9cm]{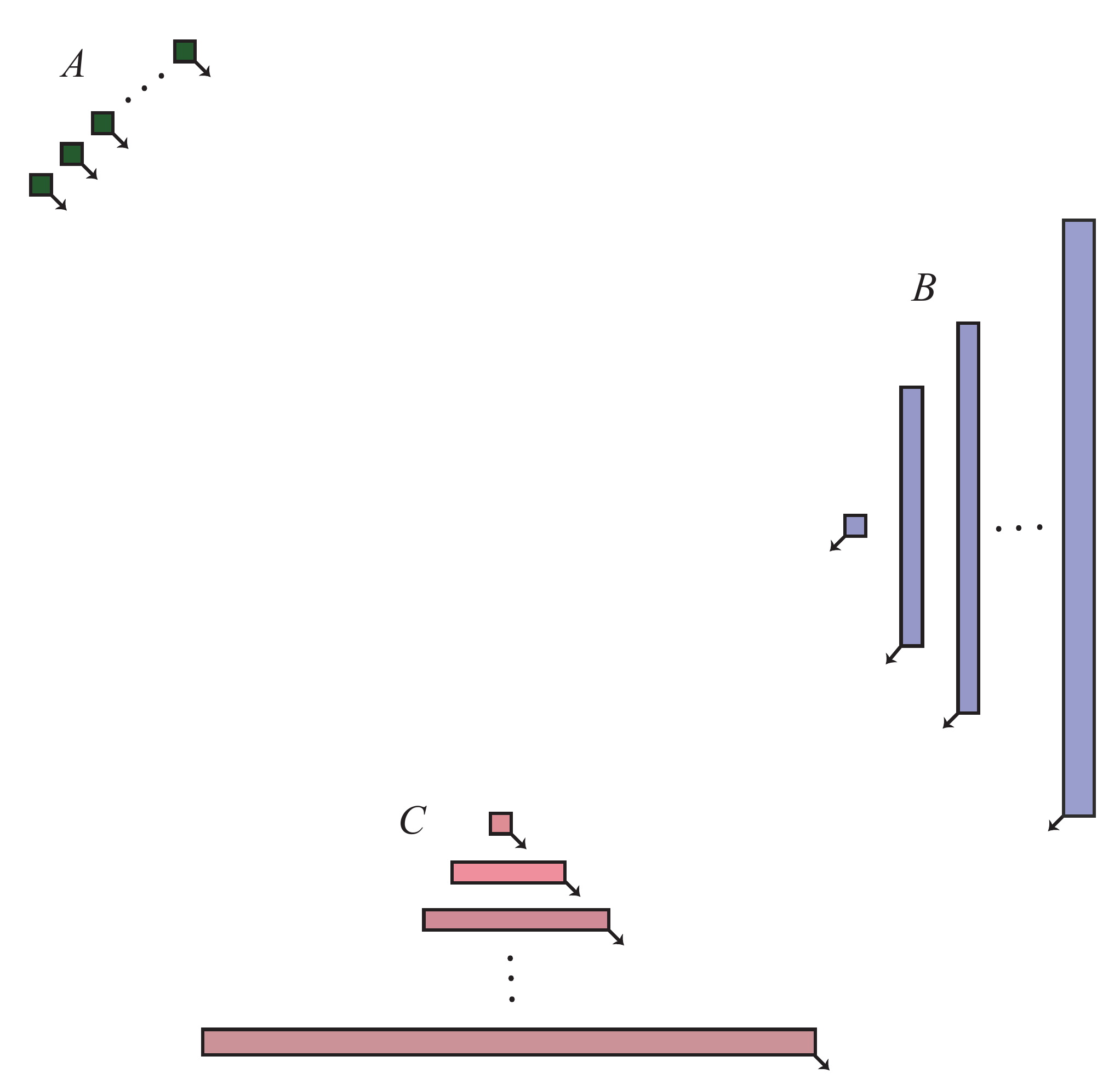}
    \caption{An SWCRV representation of a tripartite graph with a path.}
    \label{SW 3-partite}
\end{figure}

\begin{theorem}
    An SWCRVG on $n$ vertices has at most \[ \left[ \frac{n^2}{3}+\frac{n}{3}\right]-1\] edges, and the bound is tight. \label{SWCRVG-edge-bound}
\end{theorem}

\begin{proof}
    We first show that the bound is an upper bound. Since we proved that complete graphs are SCRVGs for $n<5$ (Proposition \ref{SCRVG complete graphs}) and thus SWCRVGs, the statement holds by virtue of the number of edges in the complete graph for $n<5$. We induct on $n \geq 5.$

    Let $S$ be a set of $n$ south and west rectangles and $G$ be the graph of $S.$ Suppose the bound holds for all SWCRVGs on $n-3$ vertices. Let $A$ be the rectangle of $S$ whose eye has maximal $y$-coordinate. Then there can be no rectangles in $S$ which see $A$, so $N(A)=N_+(A).$ Without loss of generality, suppose $A$ is a south rectangle. By Lemma~\ref{SWCRVG-acyclic}, there are no directed cycles in $G.$ Then there is some rectangle $B$ in $N(A)$ which is not seen by any of the other rectangles in $N(A).$ Likewise, if $N(A) \cap N(B)$ is nonempty, there is some rectangle $C$ in $N(A) \cap N(B)$ that is not seen by any rectangles in $N(A) \cap N(B).$ By Lemma \ref{monotone-out-nbhd}, $B$ and $C$ are part of an $N$-monotone sequence. 
\medskip

\noindent \textbf{Case 1: Either $N(A)$ or $N(B)$ is empty.}  In this case  $|N(A) \cap N(B) \cap N(C)| =0$.
\medskip

\noindent \textbf{Case 2: $B$ is a south rectangle.} Since $C \in N(B)$, $B$ sees $C.$ Since $B$ and $C$ are part of an $N$-monotone sequence, the north corner of $C$ cannot be visible from $B.$ Then $C$ is side-visible from $B.$ Since $N(A)$ is an $N$-monotone sequence and $B$ is a south rectangle, any rectangles in $N(A) \cap N(B)$ must be side-visible from $B$ and thus there are at most two of them by Lemma~\ref{side-viz-lemma}. Since $C$ is one of those rectangles, then there is at most one rectangle in $N(A) \cap N(B) \cap N(C).$
\medskip

\noindent \textbf{Case 3: $B$ is a west rectangle.} In this case, $B$ sees $C$ but it need not be side-visible from $B$. 
\medskip

\noindent \textbf{Case 3a: $C$ is a west rectangle.} Since $B$ and $C$ are part of an $N$-monotone sequence, any rectangle in $N(A) \cap N(B) \cap N(C)$ must be side-visible from $C.$ Thus there are at most two rectangles in  $N(A) \cap N(B) \cap N(C)$. However, there can in fact only be one, as a rectangle side-visible to its $SE$ side from $C$ blocks the sight from $A$ to $C,$ as shown in Figure \ref{Case-2a}. Thus there is at most one rectangle in $N(A) \cap N(B) \cap N(C).$ 
\medskip

\noindent \textbf{Case 3b: $C$ is a south rectangle.} If $C$ is a south rectangle, then any rectangles in $N(A) \cap N(C)$ must be side-visible from $C.$ By Lemma~\ref{side-viz-lemma} there are at most two of these rectangles. However there can only be one since a rectangle side-visible to $C$ from its $NW$ side would block sight from $B$ to $C.$
\medskip

Thus in every case $|N(A) \cap N(B) \cap N(C)| \leq 1$.  Suppose $C$ is adjacent to $k+1$ rectangles other than $A$ and $B$.  Then $k$ of those rectangles are not adjacent to at least one of $A$ and $B$.  So the number of edges incident to $a,$ $b,$ and $c$ is at most $(k+1) + 2(n-3) -k +3 = 2n-2$.

\begin{figure}
    \centering
    \includegraphics[width=5cm]{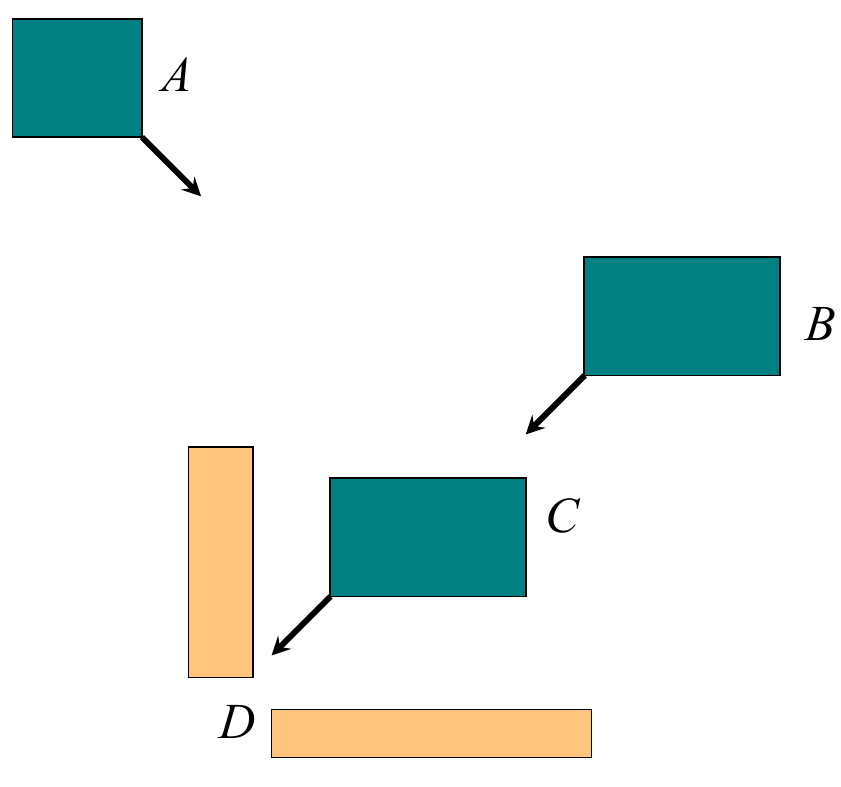}
    \caption{The two possibilities for a rectangle in $N(A) \cap N(B) \cap N(C)$ in Case 2a of Theorem~\ref{SWCRVG-edge-bound}.}
    \label{Case-2a}
\end{figure}

Thus $|E(G)| \leq |E(G-\{a,b,c\})|+2n-2.$ By the induction hypothesis, $|E(G-\{a,b,c\})| \leq \left[ \frac{(n-3)^2}{3}+\frac{n-3}{3}\right]-1.$ Thus $|E(G)| \leq \left[ \frac{(n-3)^2}{3}+\frac{n-3}{3}\right]-1+2n-2=\left[ \frac{n^2}{3}+\frac{n}{3}\right]-1.$

It remains to show the bound is tight. We give a construction as shown in Figure \ref{SW 3-partite}. It is a three-partite graph with a path in one of the parts. Label the parts $A,$ $B,$ and $C$ as shown in the figure, with $C$ being the part with the path, and suppose the parts have $a,$ $b,$ and $c$ rectangles respectively. Then there are complete bipartite graphs between each of the parts and $c-1$ edges within $C$, for a total of $ab+ac+bc+c-1$ edges. If $n=3x,$ we distribute the rectangles evenly among the three parts. If $n=3x+1,$ we put $x$ rectangles in each of $A$ and $B$ and $x+1$ rectangles in $C.$ If $n=3x+2$, we put $x$ rectangles in $A$ and $B$ and $x+2$ rectangles in $C.$ These configurations all have exactly $[\frac{n^2}{3}+\frac{n}{3}]-1$ edges.
 \end{proof}

\section{Corner Rectangle Visibility Graphs}

We now turn our attention to CRVGs which have no restriction on which direction their rectangles can look. We begin by classifying some graphs as CRVGs.

\begin{proposition}
\label{4partite}
    A complete 4-partite graph $K_{j,k,l,m}$ is a CRVG.
\end{proposition}

\begin{figure}[!ht]
      \centering
      \includegraphics[width=.6\textwidth]{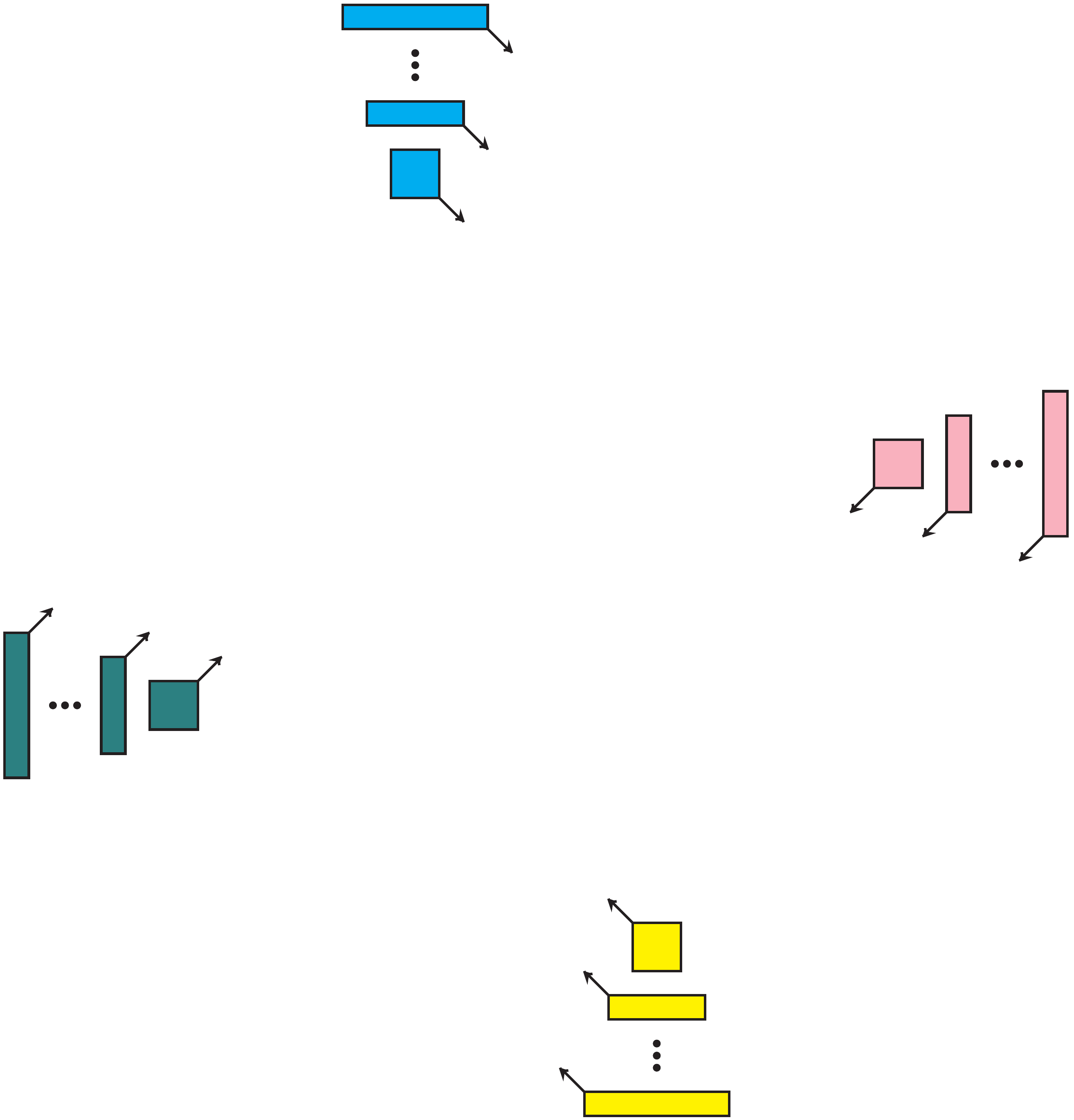}
      \caption{A CRV representation of a complete 4-partite graph.}
      \label{fig:4-partite}
\end{figure}

\begin{proof}
    Place rectangles in each part so they make a monotone sequence in two different directions. See Figure~\ref{fig:4-partite} for an example. 
\end{proof}

Note that if $n$ is divisible by 4, then the complete 4-partite graph $K_{n/4,n/4,n/4,n/4}$ with $n$ vertices has $\frac{3}{8} n^2$ edges, which is greater than the number of edges in the maximal SWCRVG with $\left [\frac{n^2}{3}+\frac{n}{3}\right]-1$ edges for all $n>1$.  Therefore $K_{n/4,n/4,n/4,n/4}$ is a CRVG that's not a SWCRVG.

\begin{proposition}
    The complete graph $K_n$ for $n \leq 5$ is a CRVG.
\end{proposition}

\begin{figure}[!ht]
      \centering
      \includegraphics[width=.6\textwidth]{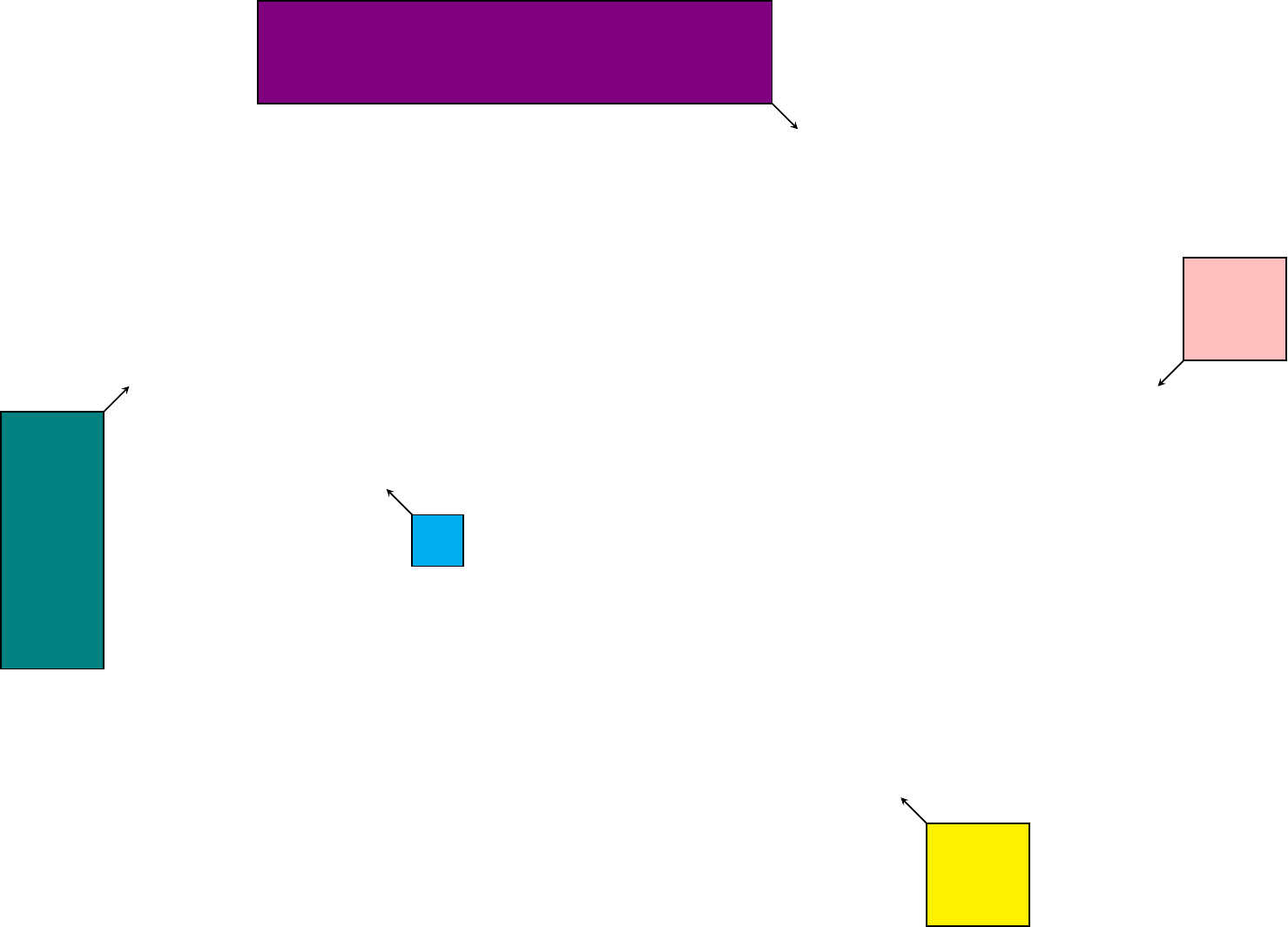}
      \caption{A CRV representation of $K_5$.}
      \label{fig:K5-CRVG}
\end{figure}

\begin{proof}
A CRV representation of $K_5$ is shown in Figure~\ref{fig:K5-CRVG}. The smaller complete graphs can be obtained by deleting rectangles from this representation.
\end{proof}

We will now work toward a proof that $K_6$ is not a CRVG.

\begin{lemma}
There is no $D$-monotone representation of $K_4.$ \label{no-monotone-K4}
\end{lemma}

\begin{proof}

Without loss of generality, let $B,C,D,$ and $E$ form a north-monotone sequence, labeled with their north corners going from left to right.

We first consider the edge $B\sim E$. This edge can happen in one of three ways: either $B_{S}(x) > D_{S}(x)$, or  $E_{W}(y) < C_{W}(y)$, or both $B_{S}(x) > C_{S}(x)$ and  $E_{W}(y) < D_{W}(y)$, as shown in Figure~\ref{fig:outdeglemmacases}.  The first two cases are analogous.

\begin{figure}[H]
\centering
        \includegraphics[width=3cm]{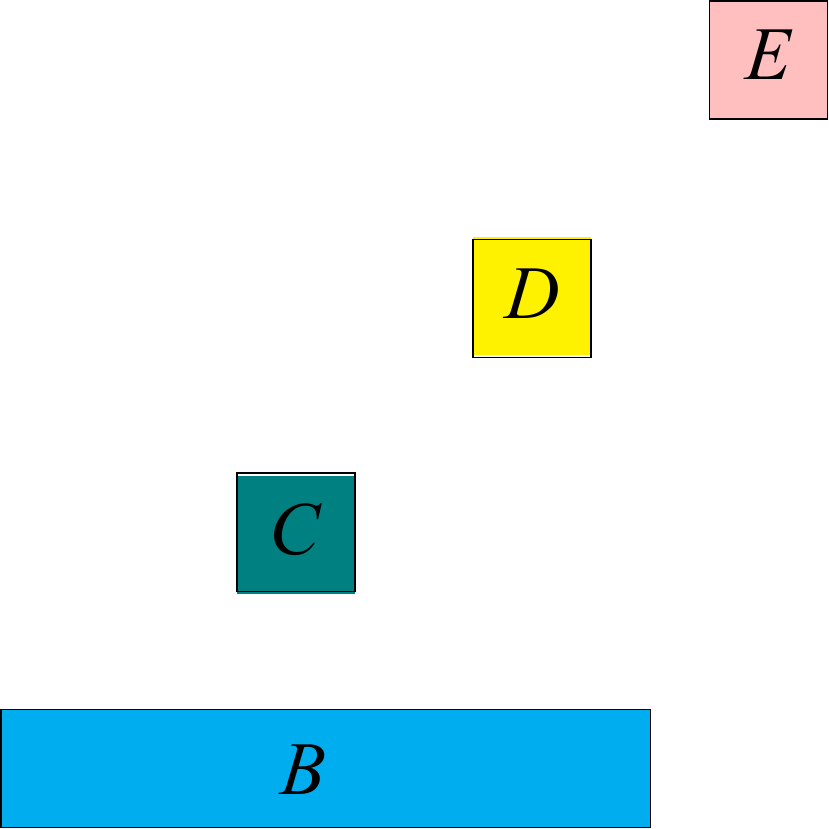} \hfill \includegraphics[width=3cm]{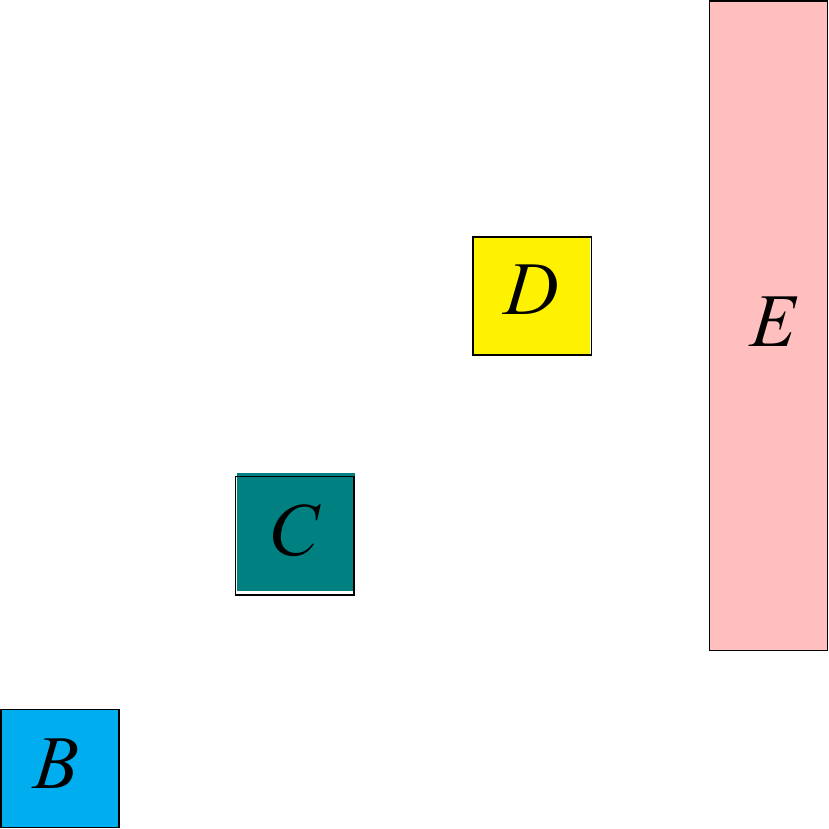}  \hfill \includegraphics[width=3cm]{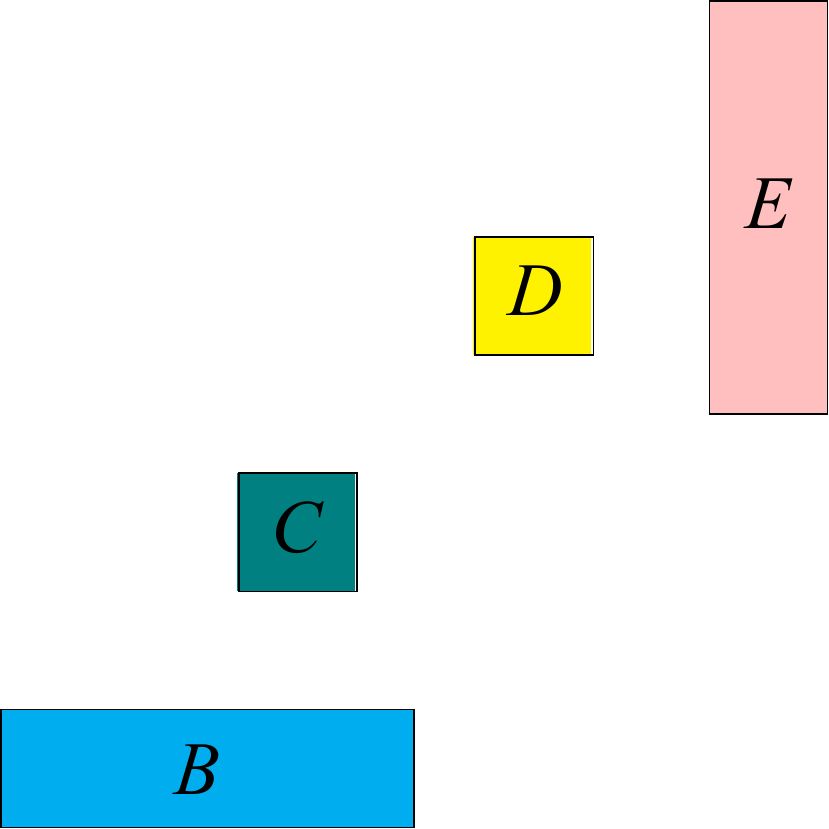}
    \caption{Three possible ways to create the $B \sim E$ edge as described in Lemma~\ref{outdegree-lemma}.}
    \label{fig:outdeglemmacases}         \label{fig:outdeglemmacase1} \label{fig:outdeglemmacase1symm}
    \label{fig:outdeglemmacase2}
\end{figure}

\bigskip

\noindent \textbf{Case 1. $B_{S}(x) > D_{S}(x)$.}  We also know that $C \sim E$.  This can happen in one of two ways: either $C_{S}(x) > D_{S}(x)$ or $E_{W}(y) < D_{W}(y)$, as shown in Figure~\ref{fig:outdeglemmasubcases}.

\begin{figure}[H]
        \centering
        \includegraphics[width=4cm]{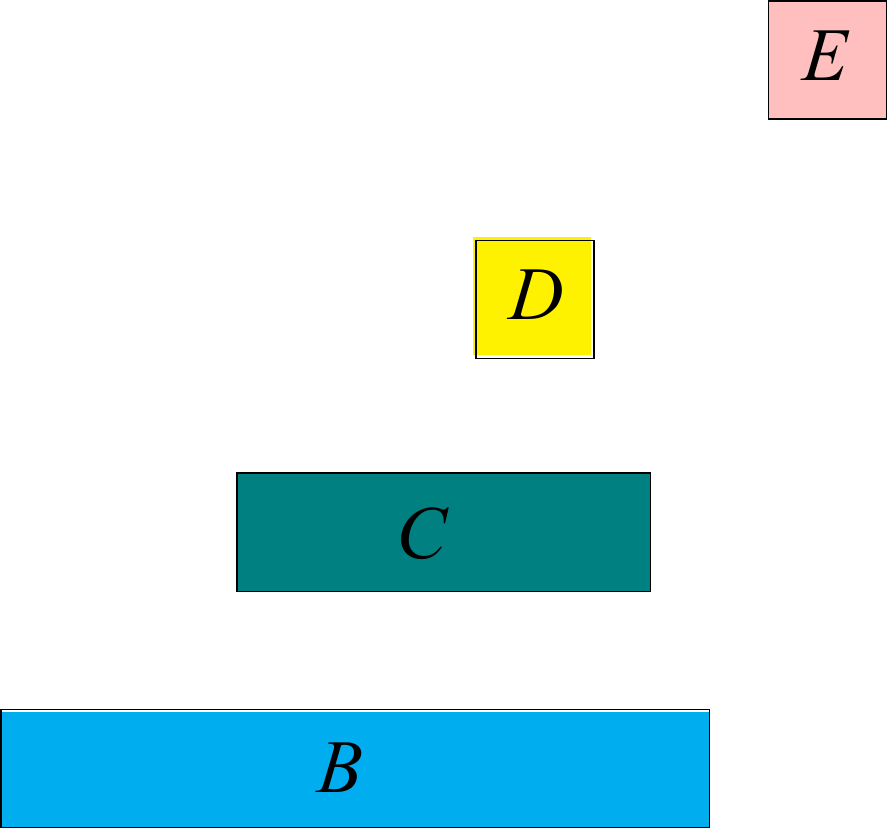} \hspace*{.2\textwidth} \includegraphics[width=4cm]{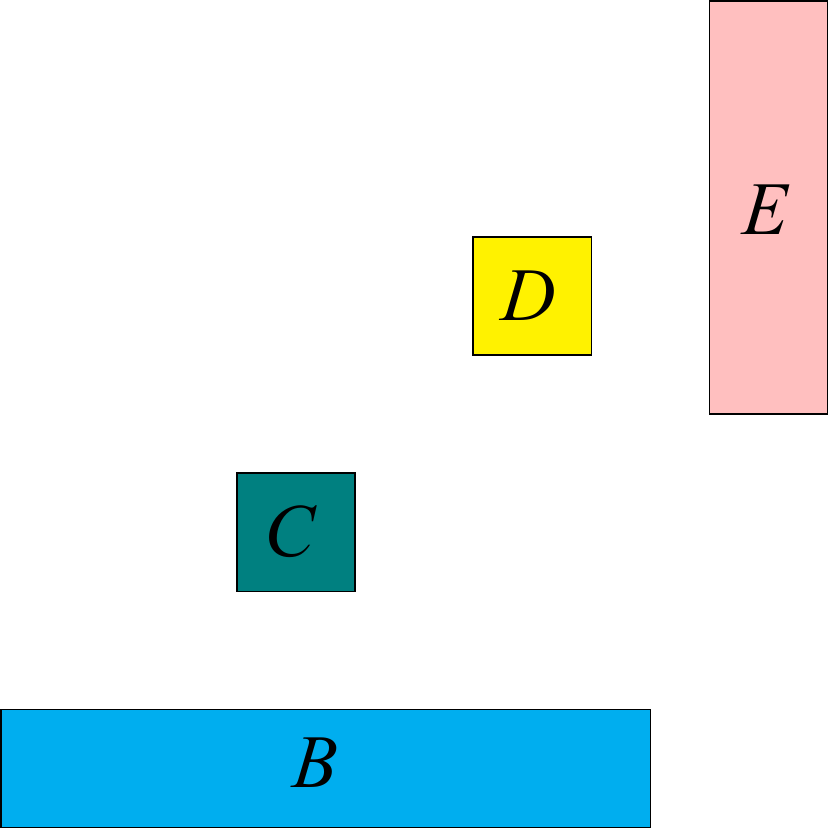}
 \caption{The subcases of Case 1 of Lemma~\ref{outdegree-lemma}.}
    \label{fig:outdeglemmasubcases} \label{fig:outdeglemmacase1b} \label{fig:outdeglemmacase1a}
\end{figure}

\bigskip

\noindent \textbf{Case 1a.  $C_{S}(x) > D_{S}(x)$.}  In this case $B$ and $D$ can't see each other.

\bigskip

\noindent \textbf{Case 1b. $E_{W}(y) < D_{W}(y)$.}  If $E$ sees $D$ then $D_{S}(y)< E_{S}(y)$. Now to have the $C \sim E$ edge we must have $C_{S}(x)>D_S(x)$ and then $D \not \sim B.$  Otherwise, $D$ sees $E$ so $D$ is a south or east rectangle.  If $D$ is an east rectangle then to have the $C \sim D$ edge, $C$ must also be an east rectangle, with $C_S(x)<D_S(x)$ in which case there can be no $B \sim C$ edge, as $C_S(x)<D_S(x)<B_S(x).$ If $D$ is a south rectangle, then either $C_S(x)>D_S(x)$ in which case there can be no $B \sim D$ edge, or $C_S(x)<D_S(x)$ and $C$ is an east rectangle, in which case there is no $B \sim C$ edge.

\bigskip

\noindent \textbf{Case 2. Both $B_{S}(x) > C_{S}(x)$ and  $E_{W}(y) < D_{W}(y)$.} Note that in this configuration, $E$ cannot see $D.$ Therefore to get the $D \sim E$ edge $D$ must see $E$ by facing either east or south. The rest is the same as in \textbf{Case 1b.} 

\bigskip

Thus it is not possible for the $N$-monotone sequence of rectangles $B,$ $C,$ $D,$ and $E$ to form a clique.
\end{proof}

\begin{corollary}
In a CRV representation of a directed graph, a rectangle in a clique of size at least 5 has out-degree at most 3 in the clique. \label{outdegree-lemma}
\end{corollary}

\begin{proof}
By Lemma~\ref{monotone-out-nbhd}, for a rectangle $A$, $N_+(A)$ is $D$-monotone.  By Lemma~\ref{no-monotone-K4}, if $N_+(A)$ is part of a clique, then $N_+(A)$ has at most 3 rectangles.
\end{proof}

\begin{lemma}
\label{K6 lemma}
    Let $S$ be a CRV representation of $K_6.$ Let $A$ and $B$ be rectangles in $S$ with out-degree $3$. Then
    \begin{enumerate}
        \item $N_+(A) \neq N_+(B),$
        \item If $A$ and $B$ are facing the same direction, then $N_+(A)$ and $N_+(B)$ are disjoint, and
         \item If $A$ and $B$ are facing opposite directions, then $N_+(A)$ and $N_+(B)$ are disjoint.
    \end{enumerate}
\end{lemma}

\begin{proof}

\noindent 1. If $N_+(A) = N_+(B)$ then $A$ and $B$ can't see each other, since neither $A$ nor $B$ can be in its own out-neighborhood.  This contradicts the fact that the graph is complete.
\medskip

\noindent 2. Suppose $A$ and $B$ are facing the same direction and their out-neighborhoods are not disjoint. Without loss of generality, suppose they are both facing south. Then by part 1, their out-neighborhoods intersect in one or two rectangles. 

Since $A$ and $B$ are facing the same 
direction, in order to have an edge between $A$ and $B$, one must be in the viewing region of the other. Say $B$ is in the viewing region of $A$, as on the left in Figure~\ref{fig:cond2-rects}.

Let $C$ be a rectangle which is in $N_+(B)$ and not in $N_+(A).$ Therefore $C$ must see $A$.  There is no way to place $C$ such that $B$ sees $C$, $C$ sees $A$, and $A$ doesn't see $C.$ An attempt to place $C$ is shown on the left in Figure~\ref{fig:cond2-rects}.

\begin{figure}[H]
    \centering
    \includegraphics[width=\textwidth]{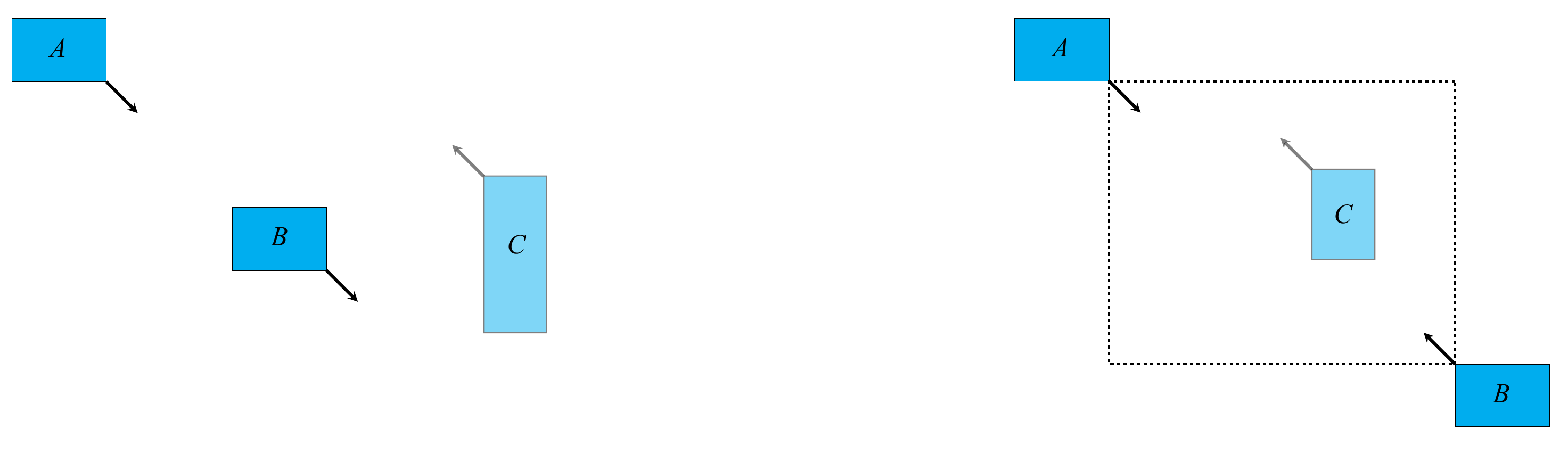}
    \caption{CRV representations of two cases of Lemma~\ref{K6 lemma}.}
    \label{fig:cond2-rects}
\end{figure}

\medskip

\noindent 3. Suppose $A$ and $B$ face opposite directions and their out-neighborhoods are not disjoint. Then they both see an additional rectangle $C$, so their viewing regions must intersect. Thus $A$ and $B$ must be placed so they can see each other, as shown on the right in Figure~\ref{fig:cond2-rects}. However, no additional rectangles can be placed in the intersection of their viewing neighborhoods without blocking the sight between $A$ and $B.$
\end{proof}

\begin{theorem}
\label{no k6}
    $K_6$ is not a CRVG.
\end{theorem}

\begin{proof}
By way of contradiction, suppose $S$ is a CRV representation of $K_6.$ By Corollary~\ref{outdegree-lemma}, a corner rectangle in $S$ has out-degree at most $3.$

As there are $15$ edges in $K_6,$ the sum of the out-degrees of its vertices is at least 15.  If it is exactly 15, the sequence of out-degrees of $K_6$ must be either $(3,3,3,2,2,2)$, $(3,3,3,3,2,1)$, or $(3,3,3,3,3,0)$. Thus at least $3$ rectangles in $S$ must have out-degree $3$, even if the total out-degree of $S$ is greater than 15. Call these $3$ rectangles $A$, $B$, and $C.$

By Lemma~\ref{K6 lemma}, if the out-neighborhoods of two rectangles with out-degree $3$ intersect, those rectangles must face in perpendicular directions.  By the Pigeonhole Principle, this implies that two of these rectangles, say $A$ and $C$, have disjoint out-neighborhoods.  Without loss of generality, the out-neighborhoods and directions of $A$, $B$, and $C$ are given by Figure~\ref{fig:no-k6-case2}.

\begin{figure}[H]
    \centering
    \includegraphics[width = 5cm]{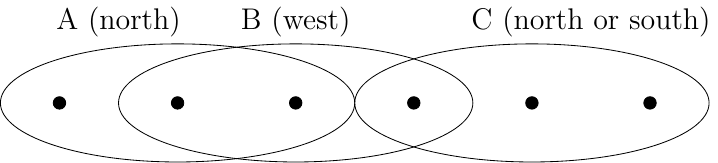}
    \caption{A Venn diagram of the out neighborhoods of $A$, $B$, and $C$.}
    \label{fig:no-k6-case2}
\end{figure}

Since no rectangle can be in its own out-neighborhood and the out-neighborhoods of $A$ and $C$ cover all the rectangles, $A$ must be in $N_+(C)$ and $C$ must be in $N_+(A)$. Thus the edge between $A$ and $C$ is bi-directed, and the total out-degree of the graph is at least $16$. Therefore a fourth rectangle $D$ must have out-degree $3$ since the possible out-degree sequences which add up to 16 are $(3,$ $3,$ $3,$ $3,$ $2,$ $2)$, and $(3,$ $3,$ $3,$ $3,$ $3,$ $1)$. 

\begin{figure}[H]
    \centering
    \includegraphics[width = 5cm]{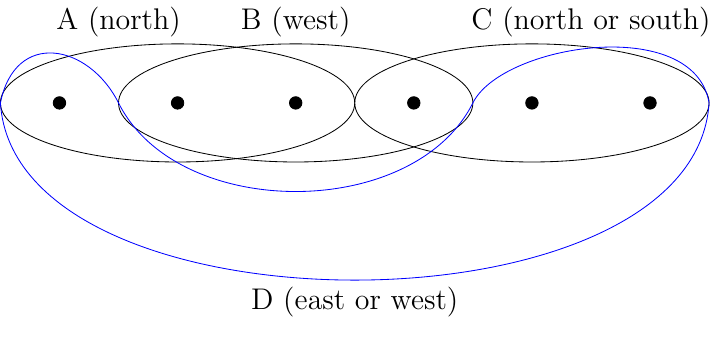}
    \caption{A Venn diagram with $A$, $B$, $C$, and $D$.}
    \label{fig:no-k6-case3}
\end{figure}

Since $N_+(D)$ must intersect $N_+(A)$ or $N_+(C)$, by Lemma~\ref{K6 lemma}, $D$ must face a direction perpendicular to $A$ or $C$ and is thus an east or west rectangle. Then $N_+(D)$ must be disjoint from $N_+(B)$. This is shown in Figure~\ref{fig:no-k6-case3}. By the same logic as before, $B$ must be in $N_+(D)$ and $D$ must be in $N_+(B).$ Thus the edge between $B$ and $D$ is bi-directed, and the total out-degree of the graph is at least $17.$ Then a fifth rectangle $E$ must have out-degree $3.$ Since $N_+(E)$ must intersect either $N_+(A)$ or $N_+(C)$, and either $N_+(B)$ or $N_+(D)$, $E$ must be either east or west and either north or south.  This is a contradiction.
\end{proof}

\begin{corollary}
    A CRVG with $6$ vertices has at most $14$ edges, a CRVG with $7$ vertices has at most $19$ edges, and a CRVG with $8$ vertices has at most $25$ edges, and these bounds are tight. \label{CRVG-678-lemma}
\end{corollary}

\begin{proof}
    Any graph that violates the above conditions contains $K_6$ as a subgraph.  In particular, if $G$ has $8$ vertices and $26$ edges, then $G$ is missing two edges, say $xy$ and $wz$ if the edges are disjoint, and $xy$ and $yz$ if they're not.  Deleting $x$ and $z$ yields a $K_6$ in both cases.  
    
    Examples of CRV representations of graphs with $6$ vertices and $14$ edges, $7$ vertices and $19$ edges, and $8$ vertices and $25$ edges are shown in Figure~\ref{fig:maximal-CRVGs}.  In this figure, the blue rectangles represent a graph with $6$ vertices and $14$ edges, the blue and orange rectangles represent a graph with $7$ vertices and $19$ edges, and all of the rectangles represent a graph with $8$ vertices and $25$ edges.
\end{proof}

\begin{figure}[!ht]
      \centering  \includegraphics[width=.4\textwidth]{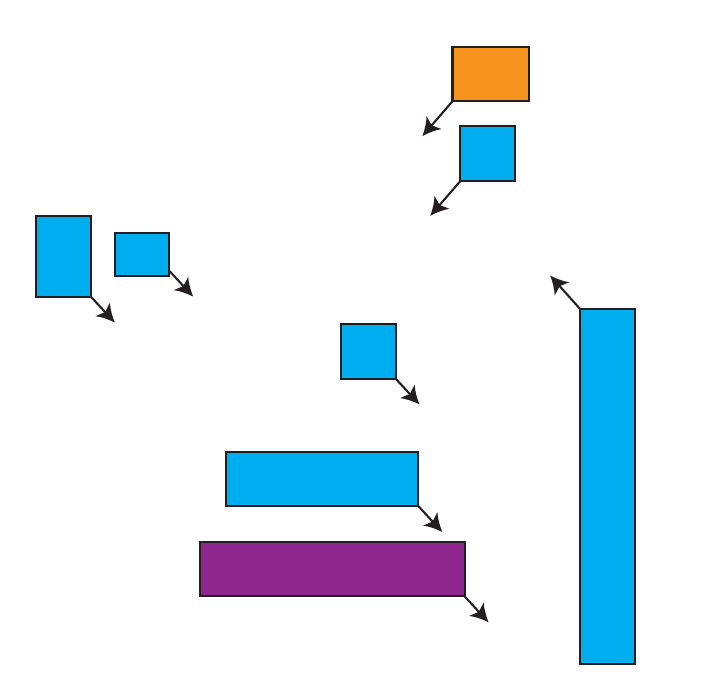}
      \caption{CRV representations of graphs with $6$ vertices and $14$ edges, $7$ vertices and $19$ edges, and $8$ vertices and $25$ edges.}
      \label{fig:maximal-CRVGs}
\end{figure}

Since we know that $K_6$ is not a CRVG and we wish to find CRVGs with as many edges as possible, Tur\'an's Theorem is useful.

\begin{theorem}[Tur\'an, 1941 \cite{westgraphtheory}]
\label{Turan}
The maximum number of edges in a graph with no $k$-clique is the number of edges in the balanced complete $(k-1)$-partite graph.
\end{theorem}

\begin{corollary}
    There is at most one CRVG on $9$ vertices with $32$ edges, and this is the maximum possible number of edges for a CRVG on $9$ vertices. \label{32-edges}
\end{corollary}

\begin{proof}
    By Tur\'an's Theorem (\ref{Turan}), a CRVG on $9$ vertices can have no more than $32$ edges.
    Let $G$ be a graph with 9 vertices and 32 edges.  Consider $\overline{G}$, the complement of $G$, which has 9 vertices and 4 edges.  No graph with 4 edges has more than one cycle.  If $\overline{G}$ has no cycles, it has 5 components, and if it has one cycle, it has 6 components. If $\overline{G}$ has 6 components, then $G$ contains a $K_6$ using one vertex from each component of $\overline{G}$.  If $\overline{G}$ has 5 components, it's a forest.  If $\overline{G}$ has 5 components including one component $C$ with more than one edge, then $G$ contains a $K_6$ using one vertex from each component of $\overline{G}$ and two from $C$.  Finally, if $\overline{G}$ has 5 components but none with more than one edge, then $G \cong K_{2,2,2,2,1}$.
\end{proof}

We don't know whether $K_{2,2,2,2,1}$ is a CRVG.  We pose this as an open question below.

\begin{proposition}
    A CRVG with the maximal number of edges on $n$ vertices has between $\lfloor\frac{3n^2}{8}\rfloor$ and $\lfloor \frac{2n^2}{5}\rfloor$ edges, which are the number of edges of the balanced complete 4-partite and balanced complete 5-partite graphs on $n$ vertices.
\end{proposition}

\begin{proof}
    By Theorem \ref{no k6}, a CRVG cannot have $K_6$ as a subgraph. Then by Tur\'{a}n's Theorem (\ref{Turan}), the greatest possible number of edges in a CRVG is at most the number of edges in the complete 5-partite graph, which is $\lfloor \frac{2n^2}{5}\rfloor.$ The complete 4-partite graph, which has $\lfloor\frac{3n^2}{8}\rfloor$ edges, is a CRVG, as given by Proposition \ref{4partite}.
\end{proof}

Note that by Lemma~\ref{CRVG-678-lemma}, A CRVG with the maximal number of edges has exactly $\lfloor \frac{2n^2}{5} \rfloor$ edges for $n=6$, $n=7$, and $n=8$, and by Corollary~\ref{32-edges}, the graph $K_{2,2,2,2,1}$ is the only CRVG with 9 vertices which might have $\lfloor \frac{2n^2}{5} \rfloor$ edges.  The construction in Figure~\ref{fig:maximal-CRVGs} yields a CRVG with $\lfloor \frac{2n^2}{5} \rfloor$ edges for $n=5$, $n=6$, $n=7$, and $n=8$, but this construction no longer has $\lfloor \frac{2n^2}{5} \rfloor$ edges for $n>8$.  For $n >14$, the construction in Figure~\ref{fig:4-partite} has more edges than the construction in Figure~\ref{fig:maximal-CRVGs}, but still might not be maximal.  If the construction in Figure~\ref{fig:4-partite} is not maximal for larger values of $n$, a new construction is needed.

\subsection{Classification of CRVGs, RVGs, and RIGs}
\label{classification}

\begin{figure}
    \centering
    \includegraphics[width=6cm]{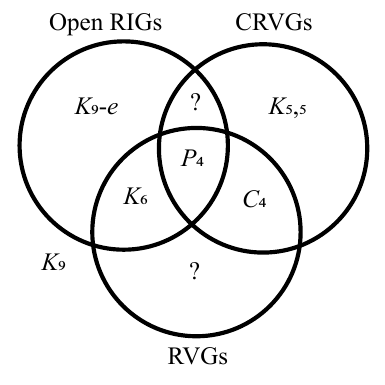}
    \caption{Classification of open RIGs, CRVGs, and RVGs}
    \label{venn_diagram}
\end{figure}

Using our previous results on CRVGs, in this subsection we classify CRVGs, RVGs, and open RIGs by containment.  We verify the example graphs shown in six of the eight categories in the Venn diagram in Figure~\ref{venn_diagram}.  We leave the remaining two categories as open questions. 

\begin{figure}
    \centering
    \includegraphics[width=4cm]{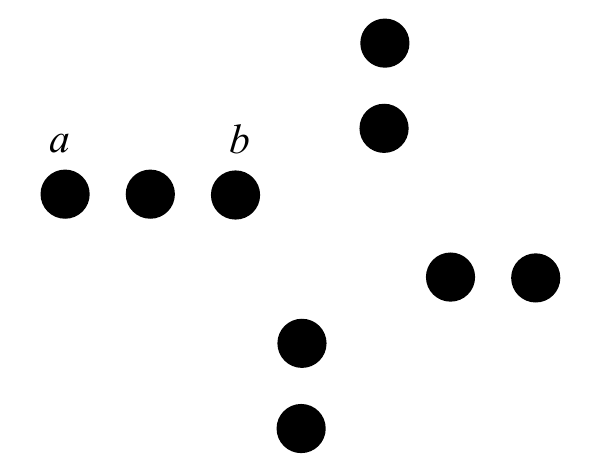}
    \caption{An open RIG representation of $K_9-e.$ The missing edge is $a \sim b.$}
    \label{fig:K9-RIG}
\end{figure}

A graph which is an open RIG but not an RVG nor a CRVG is $K_9-e.$ An open RIG representation of this graph is shown in Figure~\ref{fig:K9-RIG}.  It is not an RVG because it does not meet the edge bound of \cite{hutchinson1999representations}. It is not a CRVG because it contains $K_6$ as a subgraph (Theorem~\ref{no k6}).

A graph which is a CRVG but neither an open RIG nor an RVG is the complete bipartite graph $K_{5,5}.$ It is a CRVG by Proposition~\ref{complete_bipartite}. It is not an open RIG by \cite{liotta1998rectangle}. It is not an RVG by \cite{dean1997rectangle}.

A graph which is a CRVG and RVG but not an open RIG is the cycle on four vertices $C_4$. Cycles with more than three vertices are not open RIGs by \cite{liotta1998rectangle}. Four rectangles in a grid provides an RVG representation. A CRV representation is described in Proposition \ref{SCRVG_cycle}.

A graph which is a CRVG, an RVG, and an open RIG is simple path.

A graph which is an open RIG and an RVG but not a CRVG is $K_6.$ Theorem \ref{no k6} shows $K
_6$ is not a CRVG. Complete graphs up to $K_8$ are open RIGs by \cite{liotta1998rectangle} and RVGs by \cite{hutchinson1999representations}.
A graph which is not an open RIG, not a CRVG, and not an RVG is $K_9$ by \cite{liotta1998rectangle}, \cite{hutchinson1999representations}, and Theorem \ref{no k6}.

\section{Open Questions}

We conclude the paper with a list of open questions.

\begin{enumerate}

\item Are all closed RIGs also SCRVGs?  Is the construction of a maximal closed RIG given in~\cite{standardboxes} an SCRVG?

\item Is $K_{2,2,2,2,1}$ a CRVG?  More generally, what is the largest number of edges of a CRVG with $n$ vertices for $n>8$?

\item Filling in the last two spots in Figure~\ref{venn_diagram}, is there an example of a graph which is an open RIG and a CRVG but not an RVG?  Is there an example of a graph which is an RVG but not an open RIG and not a CRVG?

\item All planar graphs are RVGs \cite{wismath1989bar}.  Are all planar graphs CRVGs?

\item We might consider additional variations of these graphs whose rectangles can see from a fixed number of multiple corners at once.  For example, rectangles could each see from two corners, or three or four.  How many edges can a 2-corner CRVG have?  Note that if each rectangle can see from all of its corners, the representation is similar to an RIG representation, with rectangles instead of points.

\item We might also consider a variation where the number of corners from which an individual rectangle can see is unrestricted but the total number of eyes is a fixed function of the number of rectangles in a representation. For example, given $n$ rectangles what is the maximum number of edges that can be achieved by $2n$ eyes? 

\item How might these results extend to boxes in three dimensions?

\end{enumerate}

\bibliography{refs}{}
\bibliographystyle{plain}

\end{document}